\documentclass[11pt, twoside]{article}
\usepackage{amsfonts}

\usepackage{mathrsfs}
\usepackage{cite}
\usepackage{amsmath,amssymb}
\usepackage{multicol}
\usepackage{graphicx}
\usepackage{epsfig}
\usepackage{color}
\usepackage{amsthm}
\usepackage{indentfirst}
\usepackage{graphics,amsmath,amssymb,amsthm,mathrsfs,bm}
\usepackage{enumitem}

\numberwithin{equation}{section}
\newtheorem{theorem}{Theorem}[section]
\newtheorem{lemma}[theorem]{Lemma}
\newtheorem{corollary}[theorem]{Corollary}

\newtheorem{definition}[theorem]{Definition}
\newtheorem{remark}[theorem]{Remark}
\newtheorem{proposition}[theorem]{Proposition}
\allowdisplaybreaks

\makeatletter
\begin{document}

\author{Li-Jun Du and Wan-Tong Li\thanks{%
Corresponding author (wtli@lzu.edu.cn).}\\
School of Mathematics and Statistics, Lanzhou University,\\
Lanzhou, Gansu 730000, People's Republic of China
\\
Shi-Liang Wu\\
School of Mathematics and Statistics, Xidian University\\
Xi'an, Shaanxi 710071,  People's Republic of China}

\title{{Pulsating Fronts for a Bistable Lotka-Volterra Competition System with Advection in a Periodic Habitat}}

\date{ }
\maketitle

\begin{abstract}
This paper is concerned with the following Lotka-Volterra competition system with advection in a periodic habitat
\begin{equation*}
\begin{cases}
\frac{\partial u_1}{\partial t}
=d_1(x)\frac{\partial^2 u_1}{\partial x^2}-a_1(x)\frac{\partial u_1}{\partial x}+u_1\left(b_1(x)-a_{11}(x)u_1-a_{12}(x)u_2\right),\\
\frac{\partial u_2}{\partial t}
=d_2(x)\frac{\partial^2 u_2}{\partial x^2}-a_2(x)\frac{\partial u_2}{\partial x}+u_2\left(b_2(x)-a_{21}(x)u_1-a_{22}(x)u_2\right),
\end{cases}
t>0,~x\in\Bbb R,
\end{equation*}
where $d_i(\cdot)$, $a_i(\cdot)$, $b_i(\cdot)$, $a_{ij}(\cdot)$ $(i,j=1,2)$ are $L$-periodic functions
in $C^\nu(\Bbb{R})$ with some $\nu\in(0,1)$.
Under certain assumptions, the system admits two periodic locally stable steady states $(u_1^*(x),0)$ and $(0,u_2^*(x))$.
In this work, we first establish the existence of the pulsating front $\bm{U}(x,x+ct)=(U_1(x,x+ct),U_2(x,x+ct))$ connecting
two periodic solutions $(0,u_2^*(x))$ and $(u_1^*(x),0)$ at infinities.
By using a dynamical method, we confirm further that the pulsating front is asymptotically stable
for front-like initial values. As a consequence of the global asymptotically stability,
we finally show that the pulsating front is unique up to translation.



\end{abstract}

\section{Introduction}
This paper is the first one in a series of two, concerning about
the following reaction-diffusion-advection competition system in a periodic habitat
\begin{equation}\label{u1u2}
\begin{cases}
\frac{\partial u_1}{\partial t}=L_1u_1+u_1\left(b_1(x)-a_{11}(x)u_1-a_{12}(x)u_2\right),\\
\frac{\partial u_2}{\partial t}=L_2u_2+u_2\left(b_2(x)-a_{21}(x)u_1-a_{22}(x)u_2\right),
\end{cases}
t>0,~x\in\Bbb R,
\end{equation}
where $L_i:=d_i(x)\frac{\partial^2}{\partial x^2}-a_i(x)\frac{\partial}{\partial x}$,
$u_1$ and $u_2$ denote population densities of two competition species in an $L$-periodic habitat for some $L>0$.
$d_i(\cdot),\ a_i(\cdot)$ and $b_i(\cdot)$ are $L$-periodic functions in $C^\nu(\Bbb R)$ with some $\nu\in(0,1)$,
denoting the diffusion, advection and growth rates of $u_i$, respectively.
$d_i(\cdot)\geq d_0>0$, that is, $L_i$ is uniformly elliptic.
$a_{ij}(x)>0$ are $L$-periodic functions representing inter- and intraspecific competition coefficients, $i,j=1,2$.

Reaction-diffusion models of the type \eqref{u1u2} are widely used to capture the spatial dynamics of two competing species.
Of particular interest is to understand the propagation phenomena as well as spreading properties that the system exhibits.
Over the past decades, quite a few literatures have been devoted to the homogeneous cases,
mainly concerning about persistence, extinction, biological invasions, the minimal wave speeds, traveling waves and entire solutions,
see, e.g., \cite{hosono1989,kan-on1995,Lewis2002,liang2007,morita2009,lv2010,huang2011,guo2011,zhang2012} and references therein.
It is also well known that the periodic environment of space and/or time is one of the very useful approximations
to understand the influence of the environmental heterogeneity on the propagation phenomena arising
from ecological and biological processes. Whereas the homogeneous model has attracted many works in the mathematical literature,
propagation phenomena such as steady state problems, spreading speeds and traveling waves for
spatially and/or temporally heterogeneous systems like \eqref{u1u2} were studied more recently.
More specifically, in spatially heterogeneous environment, Dockery et al. \cite{Dockery1998} studied the effect of
dispersal rates on the survival of two competing species and showed that the slower diffuser always prevails.
Lou \cite{lou2006} proved that the two weakly competing species with some appropriate dispersal rates can no longer exist
in a spatially heterogenous environment.
Lam and Ni \cite{lam2012} considered the interactions between diffusion and heterogeneity of the environment of two species competition diffusion system with spatial heterogeneous growth rates in a bounded domain.
He and Ni further studied in a series of three papers \cite{he2016,he20161,he2017}
the combined effects of dispersal and spatial variations on the outcome of the competition, and
the joint effects of diffusion and spatial concentration on the global dynamics of the competition-diffusion
system with equal amount of total resources.
While Lutscher et al. \cite{Lutscher2007} added the advection term into such a
competition model and discussed spatial patterns and coexistence mechanisms for stream populations.
For a monostable semiflow in one-dimensional periodic environment,
Liang and Zhao \cite{liang2010} proved the existence of the spreading speed and its coincides with the minimal wave speed,
and Fang and Zhao \cite{fang2011} established the existence of bistable traveling fronts by interpreting the bistability
from a viewpoint of monotone dynamical systems to find a link with its monostable subsystems.
One can also see \cite{kong2015,bao2016} for spreading speeds and traveling waves for systems with nonlocal dispersal.

Recently, Yu and Zhao \cite{Yu2017} considered the propagation phenomena of reaction-diffusion-advection system \eqref{u1u2}.
They established the existence of the rightward spreading speed and its coincidence with
the minimal wave speed for spatially periodic rightward traveling waves.
Girardin \cite{Girardin2016} proved the existence of spatially periodic traveling waves for a bistable competition-diffusion system with periodic intraspecific competition coefficients.
In time heterogeneous environment, one can see \cite{zhao2011,bao2013,zhang2013,zhao2014,ma2016} for time periodic traveling waves, \cite{du2016,du20171} for other types of entire solutions, and \cite{liang2006} for spreading speeds and traveling waves of  time-periodic semiflows in one dimensional environment.
For time-space periodic environment, one can see a more recent paper \cite{fang2017} for the existence of traveling waves and
spreading speeds of time-space periodic monotone semiflows.

The current paper is the first one in a series of two, mainly devoted to the study of existence,
stability and uniqueness of pulsating fronts (see Definition \ref{def}) for the Lotka-Volterra competition system with advection in a periodic habitat. In the forthcoming paper \cite{du2017},
we are concerned with some other types of entire solutions by considering the interaction of these pulsating fronts.

Let us now make the precise assumptions. Assume that $d(\cdot),\ a(\cdot)$ and $b(\cdot)$ are $L$-periodic functions in $C^{\nu}(\Bbb R)$ and $d(\cdot)>0$ in $\Bbb{R}$. Then the periodic eigenvalue problem
\begin{equation*}
\begin{cases}
\lambda\phi=d(x)\phi^{\prime\prime}-a(x)\phi^{\prime}+b(x)\phi,\quad x\in\Bbb R,\\
\phi(x+L)=\phi(x),\quad x\in\Bbb R
\end{cases}
\end{equation*}
admits a principal eigenvalue $\lambda_0(d,a,b)$ associated with a positive $L$-periodic eigenfunction (see, e.g.,\cite{smith1995}).
Furthermore, $\lambda_0(d,a,b)$ is monotone increasing with respect to $b$ in the sense that,
if $b_1\geq b_2$ and $b_1\not\equiv b_2$ then $\lambda_0(d,a,b_1)>\lambda_0(d,a,b_2)$ (see, e.g., \cite[Lemma 15.5]{hess1991}).
Assume further that $c(\cdot)\geq,\not\equiv0$ is an $L$-periodic function in $C^{\nu}(\Bbb R)$.
Consider the following initial value problem
\begin{equation}\label{e}
\begin{cases}
\frac{\partial u}{\partial t}=d(x)\frac{\partial^2 u}{\partial {x^2}}-a(x)\frac{\partial u}{\partial x}+u(b(x)-c(x)u),\quad t>0,\ x\in\Bbb R,\\
u(0,x)=\phi(x),\quad x\in\Bbb R,
\end{cases}
\end{equation}
where $\phi$ is continuous and positive $L$-periodic in $\Bbb{R}$.
It follows from \cite[Theorem 2.3.4]{zhao2003} that if $\lambda_0(d,a,b)\leq0$,
then $u\equiv0$ is globally asymptotic stable with respect to periodic perturbations, and
if $\lambda_0(d,a,b)>0$, \eqref{e} admits a unique positive $L$-periodic solution $u^*(x)$ which is
globally asymptotic stable with respect to periodic perturbations for any nontrivial initial values.

Denote by $\Bbb P=PC(\Bbb R,\Bbb R^2)$ the set of all continuous and $L$-periodic functions from
$\Bbb R$ to $\Bbb R^2$ equipped with a norm
$\lVert\phi\rVert_{\Bbb P}=\mathop{\max}\limits_{x\in\Bbb R}\lvert\phi(x)\rvert$, then
$\Bbb P_+:=\{\phi\in \Bbb P: \phi(x)\geq0,\forall x\in\Bbb R\}$ is a closed cone of $Y$
inducing a partial ordering on $\Bbb P,$ and $(\Bbb P,\lVert\cdot\rVert_Y)$ is a Banach lattice.
Throughout the paper, we always assume that the trivial solution $(0,0)$
is linearly unstable with respect to the perturbation in $\Bbb P_+$, that is,
\begin{equation}\label{00}
\lambda_0(d_i,a_i,b_i)>0,\quad i=1,2.
\end{equation}
One should note that, by the monotonicity of $\lambda_0(d_i,a_i,b_i)$ with respect to $b_i$,
\eqref{00} holds true in particular if $b_i(\cdot)\geq,\not\equiv0$ for $i=1,2$.
By virtue of \eqref{00} and the above argument,
we know that there exists two positive $L$-periodic functions $u_1^*(x)$ and $u_2^*(x)$ satisfying respectively
$$0=L_1u_1+u_1(b_1(x)-a_{11}(x)u_1)\text{~~and~~}0=L_2u_2+u_2(b_2(x)-a_{22}(x)u_2),\ x\in\Bbb R,$$
such that $(u_1^*(x),0)$ and $(0,u_2^*(x))$ are two semitrivial periodic steady states of system \eqref{u1u2} in $\Bbb P_+$.
To impose a bistable structure on \eqref{u1u2}, we further assume
\begin{itemize}
\item[(H1)] $\lambda_0(d_1,a_1,b_1-a_{12}u_2^*)<0\text{~~and~~}\lambda_0(d_2,a_2,b_2-a_{21}u_1^*)<0$.
\end{itemize}
It then follows that $(u_1^*(x),0)$ and $(0,u_2^*(x))$ are two locally linearly stable periodic steady states
with respect to the perturbation in $\Bbb P_+$. Furthermore, system \eqref{u1u2} has at least one unstable coexistence steady state
(see, e.g., \cite[Theorem 2.4]{ferter1994}).

We would like to mention here that, under the assumptions that $\lambda_0(d_1,a_1,b_1-a_{12}u_2^*)>0$
(that is, $(0,u_2^*(x))$ is a linearly unstable steady state) and that the system has no steady state in Int$(\Bbb P_+)$,
it follows that $(u_1^*(x),0)$ is globally asymptotic stable for all initial values $(\phi_1,\phi_2)\in\Bbb P_+$
with $\phi_1\not\equiv0$ (see e.g., \cite[Theorem 2.1]{Yu2017}), namely,
system \eqref{u1u2} admits a \textbf{monostable} structure.
Yu and Zhao \cite{Yu2017} then considered the propagation phenomena of system \eqref{u1u2} with this monostable structure
and established the existence of spatially periodic traveling waves
connecting the two semitrivial periodic solutions.
In the present work, we focus system \eqref{u1u2} on the standard \textbf{bistable} structure.
To this end, we impose the following
\begin{itemize}
\item[(H2)] System \eqref{u1u2} has no stable periodic steady state in Int$(\Bbb P_+)$.
\end{itemize}
It is clear that (H2) together with (H1) confirms that system \eqref{u1u2} admits the standard bistable structure, that is,
there are two locally stable periodic semitrivial steady states $(u_1^*(x),0)$ and $(0,u_2^*(x))$, and
all the periodic intermediate (coexistence) steady states of system \eqref{u1u2} are unstable.

\begin{remark} \rm
It seems to be a very difficult but interesting problem to investigate sufficient conditions for the existence, uniqueness
and stability of periodic coexistence steady states of system \eqref{u1u2},
we leave it as an open problem, since it is not our point at present.
Indeed, we can conjecture as in \cite{Girardin2016} that the nonexistence of a stable periodic coexistence
steady state is not a necessary condition for the existence of a bistable pulsating front for system \eqref{u1u2},
it is only a required step toward the existence of pulsating fronts obtaining by using the abstract theory developed in
\cite{fang2011}.
Particularly, if we consider system \eqref{u1u2} as following
\begin{equation}\label{G}
\begin{cases}
\frac{\partial u_1}{\partial t}=\frac{\partial^2 u_1}{\partial x^2}+u_1\left(b_1(x)-a_{11}(x)u_1-a_{12}u_2\right),\\
\frac{\partial u_2}{\partial t}=d\frac{\partial^2 u_2}{\partial x^2}+u_2\left(b_2(x)-a_{21}u_1-a_{22}(x)u_2\right),
\end{cases}
t>0,~x\in\Bbb R,
\end{equation}
where $d,\ a_{12}$ and $a_{21}$ are positive constants,
$b_i(\cdot)>0$ and $a_{ii}(\cdot)>0$ are $L$-periodic functions in $C^\nu(\Bbb R)$ for some $\nu\in(0,1)$.
It then easily follows that $\lambda_0(1,0,b_1)>0$ and $\lambda_0(d,0,b_2)>0$,
and then \eqref{G} admits two semitrivial steady states (see, e.g., \cite{berestycki20051}).
Furthermore, it follows from \cite[Propositions 2.1 and 2.10]{Girardin2016} that the two semitrivial steady states are
locally linearly stable and any periodic coexistence steady state is unstable provided that $L$ is sufficiently small
and $a_{ij},\ i\neq j$ are sufficiently large, that is, (H1) and (H2) hold for \eqref{G},
which yields that system \eqref{G} admits the standard bistable structure.
\end{remark}

Let $\mathcal C$ be the set of all bounded and continuous functions from $\Bbb R$ to $\Bbb R^2$ and denote
$\mathcal C_+=\{\bm{u}\in\mathcal C: \bm{u}(x)\geq\bm{0}\text{~~for all~~} x\in\Bbb R\}$. Then $\mathcal C_+$ is a closed cone of $\mathcal C$.
We introduce a norm $\lVert\cdot\rVert_{\mathcal C}$ by $\lVert\bm{u}\rVert_{\mathcal C}=\mathop{\max}\limits_{x\in\Bbb{R}}|\bm{u}(x)|$.
For any $\bm{a}_1,\bm{a}_2\in\mathcal C$ with $\bm{a}_1<\bm{a}_2$, denote
$\mathcal C_{[\bm{a_1},\bm{a_2}]}=\{\bm{u}\in\mathcal C: \bm{a}_1\leq\bm{u}\leq\bm{a}_2\}$ and
$\mathcal C_{[\bm0,\bm{\beta}]}$ by $\mathcal C_{\bm{\beta}}$ for any $\bm{\beta}\in\rm{Int}(\mathcal C_+)$.
Hereafter, we use the usual notations for classical partial order on $\mathcal C$.
Namely, for any $\bm{u}=(u_1,u_2), \bm{v}=(v_1,v_2)\in\mathcal C$, writing $\bm{u}=\bm{v}$ if $u_1=v_1$ and $u_2=v_2$,
$\bm{u}\bm {v}=(u_1u_2,v_1v_2)$, $a\bm{u}=(au_1,au_2)$ for any constant $a$,
$\bm{u}\leq\bm{v}$ if $u_i(\cdot)\leq v_i(\cdot)$ in $\Bbb{R}$ for $i=1,2$,
$\bm{u}<\bm{v}$ if $\bm{u}\leq\bm{v}$ and $\bm{u}\neq\bm{v}$,
$\bm{u}\ll\bm{v}$ if $u_i(x)<v_i(x)$ in $\Bbb{R}$ for $i=1,2$,
and other relations are similarly to be understood componentwise.
In particular, denote by $\bm{0}=(0,0), \bm{1}=(1,1)$
and $\bm{\alpha}\in\langle\bm{0},\bm{1}\rangle$ if $\bm{0}\ll\bm{\alpha}\ll\bm{1}$.

By a change of variables
$$\tilde{u}_1(t,x)=\frac{u_1(t,x)}{u_1^*(x)},\quad\tilde{u}_2(t,x)=\frac{u_2^*(x)-u_2(t,x)}{u_2^*(x)},$$
we transform \eqref{u1u2} into the following (dropping the tilde for the convenience of writing)
\begin{equation}\label{u1u2-0-1}
\begin{cases}
\frac{\partial u_1}{\partial t}=d_1(x)\frac{\partial u_1^2}{\partial x^2}-a_1^*(x)\frac{\partial u_1}{\partial x}+f_1(x,u_1,u_2),\\
\frac{\partial u_2}{\partial t}=d_2(x)\frac{\partial u_1^2}{\partial x^2}-a_2^*(x)\frac{\partial u_2}{\partial x}+f_2(x,u_1,u_2),
\end{cases}
t>0,~x\in\Bbb R,
\end{equation}
where $L_i^*:=d_i(x)\frac{\partial^2}{\partial x^2}-a_i^*(x)\frac{\partial}{\partial x}$ is uniformly elliptic,
$a_i^*(x)=a_i(x)-2d_i(x)\frac{(u_i^*)^\prime(x)}{u_i^*(x)}$, $a_{i1}^*(x)=a_{i1}(x)u_1^*(x)$, $a_{i2}^*(x)=a_{i2}(x)u_2^*(x)$ for $i=1,2$, and
\begin{align*}
f_1(x,u_1,u_2)&=u_1\left[a_{11}^*(x)(1-u_1)-a_{12}^*(x)(1-u_2)\right],\\
f_2(x,u_1,u_2)&=(1-u_2)\left[a_{21}^*(x)u_1-a_{22}^*(x)u_2\right].
\end{align*}
It is easily seen that system \eqref{u1u2-0-1} is cooperative in the region $u_1\geq0$ and $0\leq u_2\leq 1$,
with periodic steady states $(0,1)$, $(0,0)$ and $(1,1)$.
On the other hand, if we let $\phi_i(\cdot)>0$ be positive $L$-periodic eigenfunction associated with
$\lambda_0(d_i,a_i,b_i)$, that is,
$$\lambda_0(d_i,a_i,b_i)\phi=d_i(x)\phi_i^{\prime\prime}-a_i(x)\phi_i^{\prime}+b_i(x)\phi_i,~~x\in\Bbb R,$$
then $\psi_i(x):=\frac{\phi_i(x)}{u_i^*(x)}>0$ satisfies
$$\lambda_0(d_i,a_i,b_i)\psi_i=d_i(x)\psi_i^{\prime\prime}-a_i^*(x)\psi_i^{\prime}+a_{ii}^*(x)\psi_i,\quad x\in\Bbb R,$$
and thus $\lambda_0(d_i,a_i^*,a_{ii}^*)=\lambda_0(d_i,a_i,b_i)>0$ for $i=1,2$.
As a matter of the fact, we can similarly verify that
\begin{equation}\label{i-j}
\lambda(d_i,a_i,b_i-2a_{ii}u_i^*)=\lambda(d_i,a_i^*,-a_{ii}^*),
\ \lambda(d_i,a_i^*,a_{ii}^*-a_{ij}^*)=\lambda(d_i,a_i,b_i-a_{ij}u_j^*),\ i,j=1,2, i\neq j.
\end{equation}
Therefore, one have the following equivalent assumption to (H1) and (H2) on system \eqref{u1u2-0-1}.
\begin{itemize}
\item[(A1)] $\lambda_0(d_1,a_1^*,a_{11}^*-a_{12}^*)<0$ and $\lambda_0(d_2,a_2^*,a_{22}^*-a_{21}^*)<0$.
\item[(A2)] System \eqref{u1u2-0-1} has no stable periodic steady state in $\langle\bm{0},\bm{1}\rangle$.
\end{itemize}
That is, system \eqref{u1u2-0-1} has two locally linearly stable steady states $\bm{0}$ and $\bm{1}$,
and any coexistence steady state between $\bm{0}$ and $\bm{1}$ is unstable,
where we denote $(\hat{u}_1,\hat{u}_2)$ by any periodic coexistence steady state of system \eqref{u1u2-0-1} in the sequel.
Noting that the linearized system of \eqref{u1u2-0-1} at $(\hat{u}_1,\hat{u}_2)$ is cooperative and irreducible,
and thus the Krein-Rutman theorem yields that there exists a positive $L$-periodic eigenfunction $(\hat{\phi}_1(x),\hat{\phi}_2(x))$
associated with a principal eigenvalue $\hat{\lambda}>0$ such that
\begin{equation*}
\begin{cases}
\hat{\lambda}\hat{\phi}_1=d_1\hat{\phi}_1^{\prime\prime}-a_1^*\hat{\phi}_1^{\prime}
+[a_{11}^*(x)-a_{12}^*(x)-2a_{11}^*(x)\hat{u}_1+a_{12}^*(x)\hat{u}_2]\hat{\phi}_1+a_{12}^*(x)\hat{u}_1\hat{\phi}_2,\ x\in\Bbb{R},\\
\hat{\lambda}\hat{\phi}_2=d_2\hat{\phi}_2^{\prime\prime}-a_2^*\hat{\phi}_2^{\prime}
+a_{21}^*(x)(1-\hat{u}_2)\hat{\phi}_1-[a_{22}^*(x)+a_{21}^*(x)\hat{u}_1-2a_{22}^*(x)\hat{u}_2]\hat{\phi}_2,\ x\in\Bbb{R},\\
\hat{\phi}_i(x+L)=\hat{\phi}_i(x),\ i=1,2,\ x\in\Bbb{R}.
\end{cases}
\end{equation*}
Then \eqref{u1u2-0-1} admits the standard bistable structure.
In the sequel, we shall deal with system \eqref{u1u2-0-1} under assumptions (A1) and (A2) due to its equivalence to system \eqref{u1u2}.
The definition of pulsating fronts is given as following.

\begin{definition}\label{def}
An entire solution $$\bm{u}(t,x)=\bm{U}(x,x+ct)=(U_1(x,x+ct),U_2(x,x+ct))$$
of system \eqref{u1u2-0-1} is called a leftward pulsating front with wave speed $c$, if
$\bm{U}(\cdot,\cdot+a)\in\mathcal C_{\bm{1}}$ for any $a\in\Bbb R$, and
$\bm{U}(x,\cdot)=\bm{U}(x+L,\cdot)$ for any $x\in\Bbb R$.
Moreover, we say that $\bm{U}(x,z)$ connects $\bm{0}$ to $\bm{1}$ if
$\mathop{\lim}\limits_{z\to-\infty}\lvert \bm{U}(x,z)\rvert=0$ and
$\mathop{\lim}\limits_{z\to+\infty}\lvert \bm{U}(x,z)-\bm{1}\rvert=0$ uniformly for $x\in\Bbb R$.

Similarly,
an entire solution $$\bm{v}(t,x)=\bm{V}(x,x-ct)=(V_1(x,x-ct),V_2(x,x-ct))$$ of system \eqref{u1u2-0-1}
is called a rightward pulsating front with wave speed $c$, if
$\bm{V}(\cdot,\cdot+a)\in\mathcal C_{\bm{1}}$ for any $a\in\Bbb R$, and
$\bm{V}(x,\cdot)=\bm{V}(x+L,\cdot)$ for any $x\in\Bbb R$.
Moreover, we say that $\bm{V}(x,z)$ connects $\bm{1}$ to $\bm{0}$ if
$\mathop{\lim}\limits_{z\to-\infty}\lvert \bm{V}(x,z)-\bm{1}\rvert=0$ and
$\mathop{\lim}\limits_{z\to+\infty}\lvert \bm{V}(x,z)\rvert=0$ uniformly for $x\in\Bbb R$.
\end{definition}

The rest of the paper is organized as follows.
In Section 2, we establish the existence of pulsating front of \eqref{u1u2-0-1} using the abstract theory
developed by Fang and Zhao \cite{fang2011}.
Section 3 is devoted to the construction of a pair of appropriate sub- and supersolutions.
In Section 4, we obtain the global stability and uniqueness of pulsating fronts
by using the convergence theorem for monotone semiflows (see \cite[Theorem 2.2.4]{zhao2003}),
one can also see a similar argument in \cite{xu2004,jin2008,zhang2012,bao2013,ding2014}.

\section{Existence of pulsating fronts}

In this section, we shall employ the abstract theory developed by Fang and Zhao \cite{fang2011}
to prove the existence of leftward and rightward pulsating fronts connecting $\bm{0}$ and $\bm{1}$.
Let $\bm{\beta}\in\rm{Int}(\mathcal C_+)$ and denote
$$\Pi_{\bm{\beta}}=\{\bm{u}\in\mathcal C_{\bm{\beta}}: \bm{u}(x)=\bm{u}(x+L),\ \forall x\in\Bbb R\},$$
$X=C([0,L],\Bbb{R}^2)$, $X_+=C([0,L],\Bbb{R}_+^2)$
and $X_{\bm{\beta}}=\{\bm{u}\in X: \bm{0}\leq \bm{u}\leq\bm{\beta}\}$.
Define
\begin{align*}
\mathcal{X}&=\{\bm{v}\in BC(\Bbb{R},X): \bm{v}(s)(L)=\bm{v}(s+L)(0),\ \forall s\in\Bbb{R}\},\\
\mathcal{X}_+&=\{\bm{v}\in\mathcal{X}: \bm{v}(s)\in X_+,\ \forall s\in\Bbb{R}\},\\
\mathcal{X}_{\bm{\beta}}&=\{\bm{v}\in BC(\Bbb{R},X_{\bm{\beta}}): \bm{v}(s)(L)=\bm{v}(s+L)(0),\ \forall s\in\Bbb{R}\},\\
\mathcal{K}_{\bm{\beta}}&=\{\bm{v}\in BC(L\Bbb{Z},X_{\bm{\beta}}): \bm{v}(iL)(L)=\bm{v}((i+1)L)(0),\ \forall i\in\Bbb{Z}\},
\end{align*}
where $BC(\Bbb{R},X)$ denotes the set of all continuous and bounded functions from $\Bbb{R}$ to $X$.
Let $T(t):={\rm{diag}}(T_1(t),T_2(t))$, where $T_1(t)$ and $T_2(t)$ are linear semigroups generated by
$$\frac{\partial u_1}{\partial t}=L_1^*u_1+u_1(a_{11}^*(x)-a_{12}^*(x))\text{~~and~~}
\frac{\partial u_2}{\partial t}=L_2^*u_2-a_{22}^*(x)u_2,$$
respectively.
Then $T_1(t)$ and $T_2(t)$ are compact with respect to the compact open topology for each $t>0$.
For any $\bm{u}=(u_1,u_2)\in\mathcal C_{\bm{1}}$, define $F:\mathcal C_{\bm{1}}\to\mathcal C$ by
$$F(\bm{u})=\left({\begin{array}{*{20}{c}}
-a_{11}^*(x)u_1^2+a_{12}^*(x)u_1u_2\\
a_{21}^*(x)u_1-a_{21}^*(x)u_1u_2+a_{22}^*(x)u_2^2\\
\end{array}}\right).$$
Now we can rewrite \eqref{u1u2-0-1} into the following integral form
\begin{equation}\label{Int}
\begin{cases}
\bm{u}(t)=T(t)\bm{u}(0)+\int_0^t{T(t-s)F(\bm{u}(s))ds},\ \ t>0,\\
\bm{u}(0)=\phi\in\mathcal C_{\bm{1}}.
\end{cases}
\end{equation}
Define a family of operators $\{Q_t\}_{t\geq0}$ on $\mathcal C_{\bm{1}}$ by $Q_t(\bm{\phi}):=\bm{u}(t,\cdot;\bm{\phi})$,
where $\bm{u}(t,\cdot;\bm{\phi})$ is the solution of \eqref{u1u2-0-1} with $\bm{u}(0,\cdot;\bm{\phi})=\bm{\phi}\in\mathcal C_{\bm{1}}$.
It can easily seen that for any $t>0$,
$Q_t$ is a monotone semiflow on $\mathcal C_{\bm{1}}$.
Noting also that $\{Q_t\}_{t\geq0}$ restricted on $\Pi_{\bm{1}}$ is strongly monotone in the sense that
if $\bm{\phi}>\bm{\psi}$, then $Q_t(\bm{\phi})\gg Q_t(\bm{\psi})$ for any $t>0$.
Since $T(t)$ is compact with respect to the compact open topology,
$\{Q_t\}:\mathcal C_{\bm{1}}\to\mathcal C_{\bm{1}}$ is continuous and compact with respect to the topology of the locally uniform
convergence. Therefore, assumptions (A2)-(A4) in \cite{fang2011} hold for each $Q_t$ with $t>0$.

Now consider the following two periodic eigenvalue problems
\begin{equation}\label{0-eigen}
\begin{cases}
\mu\phi_{01}=d_1\phi_{01}^{\prime\prime}-a_1^*\phi_{01}^{\prime}+(a_{11}^*(x)-a_{12}^*(x))\phi_{01},\quad x\in\Bbb R, \\
\mu\phi_{02}=d_2\phi_{02}^{\prime\prime}-a_2^*\phi_{02}^{\prime}+a_{21}^*(x)\phi_{01}-a_{22}^*(x)\phi_{02},\quad x\in\Bbb R, \\
\phi_{0i}(x)=\phi_{0i}(x+L),\ i=1,2,\quad x\in\Bbb R,
\end{cases}
\end{equation}
and
\begin{equation}\label{1-eigen}
\begin{cases}
\mu\phi_{11}=d_1\phi_{11}^{\prime\prime}-a_1^*\phi_{11}^{\prime}-a_{11}^*(x)\phi_{11}+a_{12}^*(x)\phi_{12},\quad x\in\Bbb R, \\
\mu\phi_{12}=d_2\phi_{12}^{\prime\prime}-a_2^*\phi_{12}^{\prime}+(a_{22}^*(x)-a_{21}^*(x))\phi_{12},\quad x\in\Bbb R, \\
\phi_{1i}(x)=\phi_{1i}(x+L),\ i=1,2,\quad x\in\Bbb R.
\end{cases}
\end{equation}
We introduce the following assumption.
\begin{itemize}
\item[(B1)]
$\lambda_0(d_2,a_2^*,-a_{22}^*)<\lambda_0(d_1,a_1^*,a_{11}^*-a_{12}^*)$,\quad
$\lambda_0(d_1,a_1^*,-a_{11}^*)<\lambda_0(d_2,a_2^*,a_{22}^*-a_{21}^*)$.
\end{itemize}
Assumption (B1) might be a technique assumption which ensures that
the periodic eigenvalue problem related to the linearized system of \eqref{u1u2-0-1} at $\bm{0}$ and $\bm{1}$
exactly admits a positive eigenvalue with the corresponding eigenfunction positive, respectively.
In fact, we cannot use the Krein-Rutman theorem to confirm the existence of the positive eigenfunction since
the linearized systems of \eqref{u1u2-0-1} at $\bm{0}$ and $\bm{1}$ are not irreducible, one can see a similar assumption (A3) in \cite{bao2013}.

\begin{remark}\rm
In view of \eqref{i-j}, if $d_1\equiv d_2$, $a_1\equiv a_2$, assumption (B1) holds in particular if
$$(a_{21}(x)-2a_{11}(x))u_1^*(x)\leq,\not\equiv b_2(x)-b_1(x)\leq,\not\equiv(2a_{22}(x)-a_{12}(x))u_2^*(x),\quad \forall x\in\Bbb R.$$
\end{remark}

\begin{lemma}\label{0-1-eigen-eigen}
Assume (A1)-(A2) and (B1). Then the periodic eigenvalue problems \eqref{0-eigen} and \eqref{1-eigen}
admit a eigenvalue $\mu_i^-<0$ associated with a unique positive periodic eigenfunction
$\Phi_i^*(x)=(\phi_{i1}^*(x),\phi_{i2}^*(x))$ for $i=0,1$, respectively.
\end{lemma}

\begin{proof}
For problem \eqref{0-eigen}, let $\phi_{01}^*(x)$ be the unique positive periodic eigenfunction satisfying
$\mathop{\max}\limits_{x\in[0,L]}\phi_{01}^*(x)=1$ associated with
the principal eigenvalue $\lambda_0(d_1,a_1^*,a_{11}^*-a_{12}^*)$,
that is,
$$\lambda_0(d_1,a_1^*,a_{11}^*-a_{12}^*)\phi_{01}^*
=d_1(\phi_{01}^*)^{\prime\prime}-a_1^*(\phi_{01}^*)^{\prime}+(a_{11}^*-a_{12}^*)\phi_{01}^*,
\quad\max_{x\in[0,L]}\phi_{01}^*(x)=1.$$
In terms of (B1), there holds
$\lambda_0(d_2,a_2^*,-a_{22}^*)<\lambda_0(d_1,a_1^*,a_{11}^*-a_{12}^*)$.
The same argument as in the proof of \cite[Proposition 4.2]{Yu2017} implies that there exists
a unique positive periodic function $\phi_{02}^*(x)$ such that
$\Phi_0^*(x)=(\phi_{01}^*(x),\phi_{02}^*(x))$ is the unique positive periodic eigenfunction of \eqref{0-eigen}
associated with $\mu_0^-:=\lambda_0(d_1,a_1^*,a_{11}^*-a_{12}^*)<0$.

Similarly, for problem \eqref{1-eigen}, since $\lambda_0(d_1,a_1^*,-a_{11}^*)<\lambda_0(d_2,a_2^*,a_{22}^*-a_{21}^*)$,
there exits a unique positive periodic eigenfunction
$\Phi_1^*(x)=(\phi_{11}^*(x),\phi_{12}^*(x))$ satisfying $\mathop{\max}\limits_{x\in[0,L]}\phi_{12}^*(x)=1$
associated with $\mu_1^-:=\lambda_0(d_2,a_2^*,a_{22}^*-a_{21}^*)<0$.
The proof is complete.
\end{proof}

Next we verify the strongly stability of $\bm{0}$ and $\bm{1}$ from above and below respectively,
and the strongly instability of any periodic coexistence steady state $(\hat{u}_1,\hat{u}_2)$ from both above and below,
for each $Q_t$ with $t>0$, in the sense of the following definition.

\begin{definition}[See \cite{fang2011}] A steady state $\bm{\alpha}\in\Pi_{\bm{\beta}}$ is said to be
strongly stable from below for the map $Q:\Pi_{\bm{\beta}}\to\Pi_{\bm{\beta}}$,
if there exist a positive number $\eta_0$ and a strongly positive element $\bm{e}\in\Pi_{\bm{\beta}}$ such that
$$Q(\bm{\alpha}-\eta\bm{e})\gg\bm{\alpha}-\eta\bm{e},\quad \forall \eta\in(0,\eta_0].$$
The strongly instability from below is defined by reversing the inequality.
Similarly, we can define strongly stability (instability) from above.
\end{definition}

\begin{lemma}\label{0-1}
Assume (A1)-(A2) and (B1). Then for each $Q_t$ with $t>0$, we have
\begin{itemize}
\item[(i)] $\bm{0}$ and $\bm{1}$ are strongly stable periodic steady states from above and below respectively.
\item[(ii)] Any periodic coexistence steady state $\hat{\bm{u}}=(\hat{u}_1,\hat{u}_2)$ is strongly unstable from both above and below.
\end{itemize}
\end{lemma}

\begin{proof}
(i) Noting that $\bm{0}$ is a locally stable steady state of \eqref{u1u2-0-1},
then for any $\varepsilon\in(0,1)$ satisfying
\begin{equation}\label{0-ineq}
\max\left\{\mathop{\max}\limits_{x\in[0,L]}a_{12}^*(x),
\mathop{\max}\limits_{x\in[0,L]}a_{22}^*(x)\right\}\varepsilon<-\frac{\mu_0^-}{2},
\end{equation}
there exists $\rho=\rho(\varepsilon)>0$ sufficiently small such that
for any $\bm{\phi}=(\phi_1,\phi_2)\in B_\rho(\bm{0})\cap\Bbb{R}_+^2$,
$\bm{u}(t,\cdot;\bm{\phi})\in B_\varepsilon(\bm{0})\cap\Bbb{R}_+^2$ for any $t>0$,
where $B_\sigma(\bm{\phi_0}):=\{\bm{\phi}\in\mathcal C:\lVert\bm{\phi}-\bm{\phi_0}\rVert_{\mathcal C}<\sigma\}$ for any $\sigma>0$ and  $\bm{\phi_0}\in\mathcal C$, and $\bm{u}(t,\cdot;\bm{\phi})$ is the solution of \eqref{u1u2-0-1} with initial value $\bm{\phi}$.

Let $\bm{e}:=(\phi_{01}^*,\phi_{02}^*)$. Choose some constant $\eta_0>0$ small enough such that
$\eta\bm{e}\in B_\rho(\bm{0})\cap\Bbb{R}_+^2$ for any $\eta\in(0,\eta_0]$.
It then follows from Lemma \ref{0-1-eigen-eigen} that
$$\bm{v}(t,x)=(v_1(t,x),v_2(t,x)):=\eta e^{\frac{\mu_0^-}{2}t}\bm{e}(x),\quad (t,x)\in(0,+\infty)\times\Bbb{R}$$
satisfies
\begin{equation}\label{v1-v2}
\begin{cases}
\frac{\partial v_1}{\partial t}=L_1^*v_1+\left(a_{11}^*(x)-a_{12}^*(x)-\frac{\mu_0^-}{2}\right)v_1,\quad t>0,~x\in\Bbb R,\\
\frac{\partial v_2}{\partial t}=L_2^*v_2+a_{21}^*(x)v_1-\left(a_{22}^*(x)+\frac{\mu_0^-}{2}\right)v_2,\quad t>0,~x\in\Bbb R,\\
(v_1(0,x),v_2(0,x))=\eta(\phi_{01}^*(x),\phi_{02}^*(x)),\quad x\in\Bbb R.
\end{cases}
\end{equation}
On the other hand, since $Q_t(\eta\bm{e})=\bm{u}(t,\cdot;\eta\bm{e})\in B_\varepsilon(\bm{0})\cap\Bbb{R}_+^2$ for any $t>0$,
one can see from \eqref{0-ineq} that $\bm{u}(t,\cdot;\eta\bm{e})=(u_1(t,\cdot;\eta\bm{e}),u_2(t,\cdot;\eta\bm{e}))$ satisfies
\begin{equation}\label{u1-u2}
\begin{cases}
\frac{\partial u_1}{\partial t}=L_1^*u_1+u_1\left[a_{11}^*(x)-a_{12}^*(x)-a_{11}^*(x)u_1+a_{12}^*(x)u_2\right]\\
\qquad\leq L_1^*u_1+\left(a_{11}^*(x)-a_{12}^*(x)-\frac{\mu_0^-}{2}\right)u_1,\quad t>0,~x\in\Bbb R,\\
\frac{\partial u_2}{\partial t}=L_2^*u_2+a_{21}^*(x)u_1-a_{22}^*(x)u_2-u_2\left[a_{21}^*(x)u_1-a_{22}^*(x)u_2\right]\\
\qquad\leq L_2^*u_2+a_{21}^*(x)u_1-\left(a_{22}^*(x)+\frac{\mu_0^-}{2}\right)u_2,\quad t>0,~x\in\Bbb R,\\
(u_1(0,x),u_2(0,x))=\eta(\phi_{01}^*(x),\phi_{02}^*(x)),\quad x\in\Bbb R.
\end{cases}
\end{equation}
In light of \eqref{v1-v2} and \eqref{u1-u2}, one can conclude from the comparison principle for cooperative systems
that $Q_t(\eta\bm{e})=\bm{u}(t,\cdot;\eta\bm{e})\leq\eta e^{\frac{\mu_0^-}{2}t}\bm{e}\ll\eta\bm{e}$ for any $t>0$ and
$\eta\in(0,\eta_0]$. That is, $\bm{0}$ is strongly stable from above for $Q_t$ with any $t>0$.
Similarly, we can show that $\bm{1}$ is strongly stable from below for $Q_t$ with any $t>0$.

(ii) Let $(\hat{\phi}_1,\hat{\phi}_2)>(0,0)$ be the eigenfunction corresponding to $\hat{\lambda}>0$
of the following periodic eigenvalue problem
\begin{equation}\label{coe-eigen}
\begin{cases}
\lambda\phi_1=d_1\phi_1^{\prime\prime}-a_1^*\phi_1^{\prime}
+[a_{11}^*(x)-a_{12}^*(x)-2a_{11}^*(x)\hat{u}_1+a_{12}^*(x)\hat{u}_2]\phi_1+a_{12}^*(x)\hat{u}_1\phi_2,\ x\in\Bbb{R}, \\
\lambda\phi_2=d_2\phi_2^{\prime\prime}-a_2^*\phi_2^{\prime}
+a_{21}^*(x)(1-\hat{u}_2)\phi_1-[a_{22}^*(x)+a_{21}^*(x)\hat{u}_1-2a_{22}^*(x)\hat{u}_2]\phi_2,\ x\in\Bbb{R}, \\
\phi_i(x+L)=\phi_i(x),\ i=1,2,\ x\in\Bbb{R}.
\end{cases}
\end{equation}
Denote $\hat{\bm{e}}:=(\hat{\phi}_1,\hat{\phi}_2)$ and $\eta_1=\eta_1(t):=\frac{\hat{\lambda}e^{-\frac{\hat{\lambda}}{2}t}}
{2\max\{\mathop{\max}\limits_{x\in[0,L]}a_{11}^*\hat{\phi}_1,\mathop{\max}\limits_{x\in[0,L]}a_{21}^*\hat{\phi}_1\}}$ for each $t>0$.
Let
$$\bm{w}(t,x)=(w_1(t,x),w_2(t,x)):=\hat{\bm{u}}(x)+\eta e^{\frac{\hat{\lambda}}{2}t}\hat{\bm{e}}(x),\quad (t,x)\in(0,+\infty)\times\Bbb{R},$$
where $0<\eta\leq\eta_1$ is a constant. Then direct calculation shows that
\begin{align*}
\frac{\partial w_1}{\partial t}&=L_1^*w_1+\frac{\hat{\lambda}}{2}\eta e^{\frac{\hat{\lambda}}{2}t}\hat{\phi}_1+f_1(x,\hat{u}_1,\hat{u}_2)\\
&\quad-\eta e^{\frac{\hat{\lambda}}{2}t}\left(\hat{\lambda}\hat{\phi}_1
-[a_{11}^*(x)-a_{12}^*(x)-2a_{11}^*(x)\hat{u}_1+a_{12}^*(x)\hat{u}_2]\hat{\phi}_1-a_{12}^*(x)\hat{u}_1\hat{\phi}_2\right)\\
&=L_1^*w_1+f_1(x,w_1,w_2)+\eta e^{\frac{\hat{\lambda}}{2}t}\hat{\phi}_1\left(-\frac{\hat{\lambda}}{2}+\eta e^{\frac{\hat{\lambda}}{2}t}(a_{11}^*\hat{\phi}_1-a_{12}^*\hat{\phi}_2)\right)\\
&\leq L_1^*w_1+f_1(x,w_1,w_2),\quad t>0,~x\in\Bbb R,
\end{align*}
and
\begin{align*}
\frac{\partial w_2}{\partial t}&=L_2^*w_2+\frac{\hat{\lambda}}{2}\eta e^{\frac{\hat{\lambda}}{2}t}\hat{\phi}_2+f_2(x,\hat{u}_1,\hat{u}_2)\\
&\quad-\eta e^{\frac{\hat{\lambda}}{2}t}\left(\hat{\lambda}\hat{\phi}_2-a_{21}^*(x)(1-\hat{u}_2)\hat{\phi}_1
+[a_{22}^*(x)+a_{21}^*(x)\hat{u}_1-2a_{22}^*(x)\hat{u}_2]\hat{\phi}_2\right)\\
&=L_2^*w_2+f_2(x,w_1,w_2)+\eta e^{\frac{\hat{\lambda}}{2}t}\hat{\phi}_2\left(-\frac{\hat{\lambda}}{2}+\eta e^{\frac{\hat{\lambda}}{2}t}(a_{21}^*\hat{\phi}_1-a_{22}^*\hat{\phi}_2)\right)\\
&\leq L_2^*w_2+f_2(x,w_1,w_2),\quad t>0,~x\in\Bbb R.
\end{align*}
That is, $(w_1,w_2)$ is a subsolution of \eqref{u1u2-0-1} in $(t,x)\in(0,+\infty)\times\Bbb R$.
The comparison principle then implies that
$Q_t(\hat{\bm{u}}+\eta\hat{\bm{e}})\geq\hat{\bm{u}}+\eta e^{\frac{\hat{\lambda}}{2}t}\hat{\bm{e}}\gg\hat{\bm{u}}+\eta\hat{\bm{e}}$
for any $t>0$ and $\eta\in(0,\eta_1]$, that is, $\hat{\bm{u}}$ is strongly unstable from above for each $Q_t$ with $t>0$.
Similarly, we can prove that $\hat{\bm{u}}$ is strongly unstable from below for each $Q_t$ with $t>0$. The proof is complete.
\end{proof}

Let $E$ be the set of all fixed points of $\{Q_t\}_{t\geq0}$ restricted on $\mathcal C_{\bm{1}}$,
denote by $\hat{E}$ the set of all periodic coexistence steady states,
i.e., $\hat{E}=\{\bm{\alpha}_1\in E: \bm{0}\ll\bm{\alpha}_1\ll\bm{1}\}$, then
$E=\hat{E}\cup\{\bm{0},\bm{\alpha}_2,\bm{1}\}$, where $\bm{\alpha}_2:=(0,1)$.
For any $\bm{\alpha}\in E\setminus\{\bm{0},\bm{1}\}$, in terms of Lemma \ref{0-1},
the bistable system $\{Q_t\}_{t\geq0}$ performs a monostable dynamics on
$\mathcal C_{[\bm{0},\bm{\alpha}]}$ and $\mathcal C_{[\bm{\alpha},\bm{1}]}$, respectively.
It remains to show the so-called counter-propagation assumption in \cite{fang2011}.

Introduce a family of operators $\{\hat{Q}_t\}_{t\geq0}$ on $\mathcal{X}_{\bm{1}}$ as
$$\hat{Q}_t[\bm{v}](s)(\theta):=Q_t[\bm{v}_s](\theta),
\ \ \forall \bm{v}\in\mathcal{X}_{\bm{1}},\ s\in\Bbb{R},\ \theta\in[0,L],\ t\geq0,$$
where $\bm{v}_s\in\mathcal C$ is defined by
$$\bm{v}_s(x)=\bm{v}(s+n_x)(\theta_x),\ \ \forall x=n_x+\theta_x\in\Bbb{R},\ n_x=L\left[\frac{x}{L}\right],\ \theta_x\in[0,L).$$
Then $\{\hat{Q}_t\}_{t\geq0}$ is a monotone semiflow on $\mathcal{X}_{\bm{1}}$.
Since $\{Q_t\}_{t\geq0}$ and $\{\hat{Q}_t\}_{t\geq0}$ are topologically conjecture,
the bistable system $\{\hat{Q}_t\}_{t\geq0}$ performs a monostable dynamics on
$\mathcal C_{[\bm{0},\bm{\alpha}]}$ and $\mathcal C_{[\bm{\alpha},\bm{1}]}$ respectively for any $\bm{\alpha}\in E\setminus\{\bm{0},\bm{1}\}$.
Thanks to Lemma \ref{0-1}, there exist some constants $\eta_i>0$ and
vectors $\bm{e}_i\in\rm{Int}(\mathcal{X}_+)$ for $i=1,2,$ such that
$\hat{Q}_1[\eta\bm{e}_0]\ll\eta\bm{e}_0$ for any $\eta\in(0,\eta_0]$, and
$\hat{Q}_1[\bm{1}-\eta\bm{e}_1]\gg\bm{1}-\eta\bm{e}_1$ for any $\eta\in(0,\eta_1]$.
Define
$$\theta^+:=\sup\{\theta\in[0,1]: \theta\bm{\alpha}\in\mathcal{X}_{[\bm{0},\eta_0\bm{e}_0]}\},\ \
\bm{\alpha}^+=\theta^+\bm{\alpha},$$
$$\theta^-:=\sup\{\theta\in[0,1]: \theta\bm{\alpha}+(1-\theta)\bm{1}\in\mathcal{X}_{[\bm{1}-\eta_1\bm{e}_1,\bm{1}]}\},\ \
\bm{\alpha}^-=\theta^-\bm{\alpha}+(1-\theta^-)\bm{1}.$$
As in \cite{fang2011}, we introduce
\begin{align*}
c_+^*(\bm{0},\bm{\alpha}):
&=\sup\left\{ c\in\Bbb{R}: \mathop{\lim}\limits_{n\to\infty,\ x\leq cn}\hat{Q}_1^n[\bm{\phi_{\bm{\alpha}}^+}](x)=\bm{0}\right\},\\
c_-^*(\bm{\alpha},\bm{1}):
&=\sup\left\{c\in\Bbb{R}: \mathop{\lim}\limits_{n\to\infty,\ x\geq-cn}\hat{Q}_1^n[\bm{\phi_{\bm{\alpha}}^-}](x)=\bm{1}\right\},
\end{align*}
where the initial functions $\bm{\phi_{\bm{\alpha}}^\pm}\in\mathcal{X}_{\bm{1}}$ satisfy respectively
$$\bm{\phi}_{\bm{\alpha}}^+(x)=\bm{\alpha},\ \forall x\geq1,\text{~~and~~}\bm{\phi}_{\bm{\alpha}}^+(x)=\bm{v_{\bm{\alpha}}^+},
\ \forall x\leq0,$$
$$\bm{\phi}_{\bm{\alpha}}^-(x)=\bm{\alpha},\ \forall x\leq-1,\text{~~and~~}\bm{\phi}_{\bm{\alpha}}^-(x)=\bm{v_{\bm{\alpha}}^-},
\ \forall x\geq0.$$
In the following, we are going to show $c_+^*(\bm{0},\bm{\alpha})+c_-^*(\bm{\alpha},\bm{1})>0$
for any $\bm{\alpha}\in E\setminus\{\bm{0},\bm{1}\}$, which nearly assures the propagation of a bistable pulsating front.

For any $\bm{\alpha}_1=(\hat{u}_1,\hat{u}_2)\in\hat{E}$,
to better understand the rightward propagation dynamics of $\{Q_t\}_{t\geq0}$ restricted on $\mathcal C_{[\bm{0},\bm{\alpha}_1]}$,
let $\tilde{v}_i(t,x)=\hat{u}_i(x)-u_i(t,x)$, $i=1,2$,
then this dynamics is equivalent to that of the following system restricted on $\mathcal C_{\bm{\alpha}_1}$:
\begin{equation}\label{u1u2-linearize+}
\begin{cases}
\frac{\partial \tilde{v}_1}{\partial t}=L_1^*\tilde{v}_1
+f_1(x,\hat{u}_1(x),\hat{u}_2(x))-f_1(x,\hat{u}_1(x)-\tilde{v}_1(t,x),\hat{u}_2(x)-\tilde{v}_2(t,x)),\\
\frac{\partial \tilde{v}_2}{\partial t}=L_2^*\tilde{v}_2
+f_2(x,\hat{u}_1(x),\hat{u}_2(x))-f_2(x,\hat{u}_1(x)-\tilde{v}_1(t,x),\hat{u}_2(x)-\tilde{v}_2(t,x)),
\end{cases}
\end{equation}
where $\bm{0}$ is unstable and $\bm{\alpha}_1\gg\bm{0}$ is stable.
Define a family of operators $\{\tilde{Q}_t\}_{t\geq0}$ on $\mathcal C_{\bm{\alpha}_1}$ by
$\tilde{Q}_t(\bm{\phi}):=\bm{\tilde{v}}(t,\cdot;\bm{\phi})$,
where $\bm{\tilde{v}}(t,\cdot;\bm{\phi})$ is the solution of \eqref{u1u2-linearize+} with
$\bm{\tilde{v}}(0,\cdot;\bm{\phi})=\bm{\phi}\in\mathcal C_{\bm{\alpha}_1}$.
Then $\tilde{Q}_t(\bm{0})=\bm{0}$, $\tilde{Q}_t(\bm{\alpha}_1)=\bm{\alpha}_1$ for all $t\geq0$, and
$\tilde{Q}_t$ satisfies all hypotheses (E1)-(E5) in \cite{liang2010} for any $t>0$.
It then follows from \cite[Theorem 5.1]{liang2010} that
$\tilde{Q}_1:\mathcal C_{\bm{\alpha}_1}\to\mathcal C_{\bm{\alpha}_1}$ admits
the rightward spreading speed $c_+^{**}(\bm{0},\bm{\alpha}_1)$, and then
$c_+^{*}(\bm{0},\bm{\alpha}_1)\geq c_+^{**}(\bm{0},\bm{\alpha}_1)$.

Similarly, to better understand the leftward propagation dynamics of $\{Q_t\}_{t\geq0}$ restricted on $\mathcal C_{[\bm{\alpha}_1,\bm{1}]}$,
let $\tilde{u}_i(t,x)=u_i(t,x)-\hat{u}_i(x)$, $i=1,2$,
then this dynamics is equivalent to that of the following system restricted on $\mathcal C_{\bm{\beta}}$:
\begin{equation}\label{u1u2-linearize-}
\begin{cases}
\frac{\partial \tilde{u}_1}{\partial t}=L_1^*\tilde{u}_1
+f_1(x,\tilde{u}_1(t,x)+\hat{u}_1(x),\tilde{u}_2(t,x)+\hat{u}_2(x))-f_1(x,\hat{u}_1(x),\hat{u}_2(x)),\\
\frac{\partial \tilde{u}_2}{\partial t}=L_2^*\tilde{u}_2
+f_2(x,\tilde{u}_1(t,x)+\hat{u}_1(x),\tilde{u}_2(t,x)+\hat{u}_2(x))-f_2(x,\hat{u}_1(x),\hat{u}_2(x)),
\end{cases}
\end{equation}
where $\bm{0}$ is unstable and $\bm{\beta}:=\bm{1}-\bm{\alpha}_1\gg\bm{0}$ is stable.
Define a family of operators $\{\tilde{P}_t\}_{t\geq0}$ on $\mathcal C_{\bm{\beta}}$ by
$\tilde{P}_t(\bm{\phi}):=\bm{\tilde{u}}(t,\cdot;\bm{\phi})$,
where $\bm{\tilde{u}}(t,\cdot;\bm{\phi})$ is the solution of \eqref{u1u2-linearize-} with
$\bm{\tilde{u}}(0,\cdot;\bm{\phi})=\bm{\phi}\in\mathcal C_{\bm{\beta}}$.
Noting that $\tilde{P}_t(\bm{0})=\bm{0}$, $\tilde{P}_t(\bm{\beta})=\bm{\beta}$ for all $t\geq0$, and
$\tilde{P}_t$ satisfies all hypotheses (E1)-(E5) in \cite{liang2010} for any $t>0$.
Then $\tilde{P}_1:\mathcal C_{\bm{\beta}}\to\mathcal C_{\bm{\beta}}$ admits
the leftward spreading speed $c_-^{**}(\bm{0},\bm{\beta})$, and then
$c_-^*(\bm{\alpha}_1,\bm{1})\geq c_-^{**}(\bm{0},\bm{\beta})$.

Next we use the linear operators approach (see \cite{liang2007,weinberger2003})
to estimate these two lower spreading speeds $c_+^{**}(\bm{0},\bm{\alpha}_1)$ and $c_-^{**}(\bm{0},\bm{\beta})$.

\begin{lemma}\label{coun-1}
Assume (A1)-(A2) and (B1). Then $c_+^{**}(\bm{0},\bm{\alpha}_1)+c_-^{**}(\bm{0},\bm{\beta})>0$.
\end{lemma}
\begin{proof}
Consider the linearized system of \eqref{u1u2-linearize+} at $\bm{0}$:
\begin{equation}\label{lin-1}
\begin{cases}
\frac{\partial v_1}{\partial t}=L_1^*v_1
+\frac{\partial f_1}{\partial u_1}(x,\hat{u}_1(x),\hat{u}_2(x))v_1
+\frac{\partial f_1}{\partial u_2}(x,\hat{u}_1(x),\hat{u}_2(x))v_2,\ \ t>0,\ x\in\Bbb{R},\\
\frac{\partial v_2}{\partial t}=L_2^*v_2
+\frac{\partial f_2}{\partial u_1}(x,\hat{u}_1(x),\hat{u}_2(x))v_1
+\frac{\partial f_2}{\partial u_2}(x,\hat{u}_1(x),\hat{u}_2(x))v_2,\ \ t>0,\ x\in\Bbb{R},\\
\bm{v}(0,\cdot)=\bm{\phi}\in\mathcal C.
\end{cases}
\end{equation}
Let $\{M(t)\}_{t\geq0}:\mathcal C_{\bm{\alpha}_1}\to\mathcal C_{\bm{\alpha}_1}$ be the solution map of \eqref{lin-1},
that is, $M(t)(\bm{\phi}):=\bm{v}(t,\cdot;\bm{\phi})$,
where $\bm{v}(t,\cdot;\bm{\phi})$ is the solution of \eqref{lin-1} with $\bm{v}(0,\cdot;\bm{\phi})=\bm{\phi}$.
For any given $\mu\in\Bbb{R}$, letting $\bm{v}(t,x)=e^{-\mu x}\bm{w}(t,x)$ in \eqref{lin-1}, then
$\bm{w}(t,x)=(w_1(t,x),w_2(t,x))$ satisfies the following $\mu$-parameterized linear system
\begin{equation}\label{lin-1-v}
\begin{cases}
\frac{\partial w_1}{\partial t}=L_1^*w_1-2\mu d_1\frac{\partial w_1}{\partial x}
+\left(d_1\mu^2+a_1^*\mu+\frac{\partial f_1}{\partial u_1}(x,\hat{u}_1,\hat{u}_2)\right)w_1
+\frac{\partial f_1}{\partial u_2}(x,\hat{u}_1,\hat{u}_2)w_2,\\
\frac{\partial w_2}{\partial t}=L_2^*w_2-2\mu d_2\frac{\partial w_2}{\partial x}
+\frac{\partial f_2}{\partial u_1}(x,\hat{u}_1,\hat{u}_2)w_1
+\left(d_2\mu^2+a_2^*\mu+\frac{\partial f_2}{\partial u_2}(x,\hat{u}_1,\hat{u}_2)\right)w_2,\\
\bm{w}(0,x)=\bm{\phi}(x)e^{\mu x}.
\end{cases}
\end{equation}
Let $\{M_{\mu}(t)\}_{t\geq0}$ be the solution map of system \eqref{lin-1-v},
that is, $M_{\mu}(t)(\bm{\phi}):=\bm{w}(t,\cdot;\bm{\phi})$,
where $\bm{w}(t,\cdot;\bm{\phi})$ is the unique solution of \eqref{lin-1-v} with $\bm{w}(0,\cdot;\bm{\phi})=\bm{\phi}$.
Noting that system \eqref{lin-1-v} is cooperative and irreducible,
it follows that $\{M_{\mu}(t)\}_{t\geq0}$ is strongly order preserving and
$$M_{\mu}(t)[\bm{\phi}](x)=e^{\mu x}M(t)[e^{-\mu x}\bm{\phi}](x),\ \forall t\geq0,\ x\in\Bbb{R},\ \bm{\phi}\in\mathcal C.$$
Substituting $\bm{w}(t,x)=e^{\lambda t}\bm{\phi}(x)$ into \eqref{lin-1-v},
we obtain the following periodic eigenvalue problem
\begin{equation}\label{ei}
\begin{cases}
\lambda\phi_1=d_1\phi_1^{\prime\prime}-(2d_1\mu+a_1^*)\phi_1^{\prime}
+\left(d_1\mu^2+a_1^*\mu+\frac{\partial f_1}{\partial u_1}(x,\hat{u}_1,\hat{u}_2)\right)\phi_1
+\frac{\partial f_1}{\partial u_2}(x,\hat{u}_1,\hat{u}_2)\phi_2,\\
\lambda\phi_2=d_2\phi_2^{\prime\prime}-(2d_2\mu+a_2^*)\phi_2^{\prime}
+\frac{\partial f_2}{\partial u_1}(x,\hat{u}_1,\hat{u}_2)\phi_1
+\left(d_2\mu^2+a_2^*\mu+\frac{\partial f_2}{\partial u_2}(x,\hat{u}_1,\hat{u}_2)\right)\phi_2,\\
\phi_i(x+L)=\phi_i(x),\ i=1,2.
\end{cases}
\end{equation}
Let $\lambda^+=\lambda^+(\mu)$ be the principal eigenvalue associated with a positive $L$-periodic eigenfunction
for the linear periodic cooperative and irreducible system \eqref{ei}, then a similar argument as in the proof of
\cite[Lemma 3.7]{liang2007} shows that $\lambda^+(\mu)$ is a convex function on $\mu\in\Bbb{R}$.
Furthermore, $\lambda^+(0)=\hat{\lambda}>0$, which together with \eqref{ei} yields that
$\mathop{\lim}\limits_{\mu\to0^+}\frac{\lambda^+(\mu)}{\mu}=\mathop{\lim}\limits_{\mu\to+\infty}\frac{\lambda^+(\mu)}{\mu}=+\infty$,
and hence $\frac{\lambda^+(\mu)}{\mu}$ attains its infimum at some $\mu\in(0,+\infty)$.
Similar observations hold for $\frac{\lambda^+(-\mu)}{\mu}$.

For any given $\varepsilon\in(0,1)$, let $\lambda^+_\varepsilon(\mu)$ be the principal eigenvalue of the following
linear periodic cooperative and irreducible system
\begin{equation*}
\begin{cases}
\lambda\phi_1=d_1\phi_1^{\prime\prime}-(2d_1\mu+a_1^*)\phi_1^{\prime}
+\left(d_1\mu^2+a_1^*\mu+\frac{\partial f_1}{\partial u_1}(x,\hat{u}_1,\hat{u}_2)-\varepsilon\right)\phi_1
+\frac{\partial f_1}{\partial u_2}(x,\hat{u}_1,\hat{u}_2)\phi_2,\\
\lambda\phi_2=d_2\phi_2^{\prime\prime}-(2d_2\mu+a_2^*)\phi_2^{\prime}
+\frac{\partial f_2}{\partial u_1}(x,\hat{u}_1,\hat{u}_2)\phi_1
+\left(d_2\mu^2+a_2^*\mu+\frac{\partial f_2}{\partial u_2}(x,\hat{u}_1,\hat{u}_2)-\varepsilon\right)\phi_2,\\
\phi_i(x+L)=\phi_i(x),\ i=1,2,
\end{cases}
\end{equation*}
and let $\{M_\varepsilon(t)\}_{t\geq0}$ be the solution map of the linear periodic system
\begin{equation*}
\begin{cases}
\frac{\partial \tilde{v}_1}{\partial t}=L_1^*\tilde{v}_1
+\left(\frac{\partial f_1}{\partial u_1}(x,\hat{u}_1(x),\hat{u}_2(x))-\varepsilon\right)\tilde{v}_1
+\frac{\partial f_1}{\partial u_2}(x,\hat{u}_1(x),\hat{u}_2(x))\tilde{v}_2,\ \ t>0,\ x\in\Bbb{R},\\
\frac{\partial \tilde{v}_2}{\partial t}=L_2^*\tilde{v}_2
+\frac{\partial f_2}{\partial u_1}(x,\hat{u}_1(x),\hat{u}_2(x))\tilde{v}_1
+\left(\frac{\partial f_2}{\partial u_2}(x,\hat{u}_1(x),\hat{u}_2(x))-\varepsilon\right)\tilde{v}_2,\ \ t>0,\ x\in\Bbb{R},\\
\bm{\tilde{v}}(0,\cdot)=\bm{\phi}\in\mathcal C.
\end{cases}
\end{equation*}
Let a vector $\bm{\sigma}\in\Bbb{R}_+^2$ be such that
for any $(\tilde{v}_1(t,x),\tilde{v}_2(t,x))\in[\bm{0},\bm{\sigma}]$ and $x\in[0,L]$, there is
$$|a_{12}^*(x)\tilde{v}_2(t,x)-a_{11}^*(x)\tilde{v}_1(t,x)|\leq\varepsilon
\text{~~and~~}|a_{22}^*(x)\tilde{v}_2(t,x)-a_{21}^*(x)\tilde{v}_1(t,x)|\leq\varepsilon.$$
Then we have
\begin{align*}
&f_1(x,\hat{u}_1(x),\hat{u}_2(x))-f_1(x,\hat{u}_1(x)-\tilde{v}_1(t,x),\hat{u}_2(x)-\tilde{v}_2(t,x))\\
&\quad-\frac{\partial f_1}{\partial u_1}(x,\hat{u}_1(x),\hat{u}_2(x))\tilde{v}_1(t,x)-
\frac{\partial f_1}{\partial u_2}(x,\hat{u}_1(x),\hat{u}_2(x))\tilde{v}_2(t,x)\\
&=a_{11}^*(x)\tilde{v}_1^2(t,x)-a_{12}^*(x)\tilde{v}_1(t,x)\tilde{v}_2(t,x)\\
&\geq-\varepsilon\tilde{v}_1(t,x),
\end{align*}
and
\begin{align*}
&f_2(x,\hat{u}_1(x),\hat{u}_2(x))-f_2(x,\hat{u}_1(x)-\tilde{v}_1(t,x),\hat{u}_2(x)-\tilde{v}_2(t,x))\\
&\quad-\frac{\partial f_2}{\partial u_1}(x,\hat{u}_1(x),\hat{u}_2(x))\tilde{v}_1(t,x)-
\frac{\partial f_2}{\partial u_2}(x,\hat{u}_1(x),\hat{u}_2(x))\tilde{v}_2(t,x)\\
&=a_{21}^*(x)\tilde{v}_1(t,x)\tilde{v}_2(t,x)-a_{22}^*(x)\tilde{v}_2^2(t,x)\\
&\geq-\varepsilon\tilde{v}_2(t,x).
\end{align*}
Since $\bm{\tilde{v}}(t,x;\bm{0})=\bm{0}$,
there exists a positive vector $\bm{\eta}\in\Bbb{R}_+^2$ with $\bm{\eta}\leq\bm{\alpha}_1$ such that
for any $\bm{\phi}\in\mathcal C_{\bm{\eta}}$,
$\bm{\tilde{v}}(t,x;\bm{\phi})\in[\bm{0},\bm{\sigma}]$ for any $x\in\Bbb{R}$ and $t\in[0,1]$,
where $\bm{\tilde{v}}(t,x;\bm{\phi})$ is the unique solution of \eqref{u1u2-linearize+}.
By the comparison principle, there holds
$$\tilde{Q}_t(\bm{\phi})\geq M_\varepsilon(t)(\bm{\phi}),\ \forall\bm{\phi}\in\mathcal C_{\bm{\eta}},\ t\in[0,1].$$
In particular, $\tilde{Q}_1(\bm{\phi})\geq M_\varepsilon(1)(\bm{\phi})$ for any $\bm{\phi}\in\mathcal C_{\bm{\eta}}$.
It follows from \cite[Theorem 2.4]{weinberger2003} that
$$c_+^{**}(\bm{0},\bm{\alpha}_1)\geq\mathop{\inf}\limits_{\mu>0}\frac{\lambda^+_\varepsilon(\mu)}{\mu},\ \ \forall\varepsilon\in(0,1).$$
Letting $\varepsilon\to0$ in the above inequality, we have
$$c_+^{**}(\bm{0},\bm{\alpha}_1)\geq\mathop{\inf}\limits_{\mu>0}\frac{\lambda^+(\mu)}{\mu}.$$
By the change of variable $\bm{\tilde{w}}(t,x)=\bm{\tilde{u}}(t,-x)$ in system \eqref{u1u2-linearize-},
it follows that $c_-^{**}(\bm{0},\bm{\beta})$ is the rightward spreading speed of the resulting system for $\bm{\tilde{w}}$.
The similar argument as above implies that
$$c_-^{**}(\bm{0},\bm{\beta})\geq\mathop{\inf}\limits_{\mu>0}\frac{\lambda^+(-\mu)}{\mu}.$$
Let $\mu_1>0$ and $\mu_2>0$ be such that
$$c_+^{**}(\bm{0},\bm{\alpha}_1)\geq\mathop{\inf}\limits_{\mu>0}\frac{\lambda^+(\mu)}{\mu}=\frac{\lambda^+(\mu_1)}{\mu_1},\quad
c_-^{**}(\bm{0},\bm{\beta})\geq\mathop{\inf}\limits_{\mu>0}\frac{\lambda^+(-\mu)}{\mu}=\frac{\lambda^+(-\mu_2)}{\mu_2}.$$
Define $\theta:=\frac{\mu_1}{\mu_1+\mu_2}\in(0,1)$, then $\theta\mu_1+(1-\theta)(-\mu_2)=0$.
The convexity of $\lambda^+(\mu)$ then shows that
\begin{align*}
c_+^{**}(\bm{0},\bm{\alpha}_1)+c_-^{**}(\bm{0},\bm{\beta})&\geq\frac{\lambda^+(\mu_1)}{\mu_1}+\frac{\lambda^+(-\mu_2)}{\mu_2}
=\frac{\mu_1+\mu_2}{\mu_1\mu_2}[\theta\lambda^+(\mu_1)+(1-\theta)\lambda^+(-\mu_2)]\\
&\geq\frac{\mu_1+\mu_2}{\mu_1\mu_2}[\lambda^+(\theta\mu_1+(1-\theta)(-\mu_2))]\\
&=\frac{\mu_1+\mu_2}{\mu_1\mu_2}\lambda^+(0)=\frac{\mu_1+\mu_2}{\mu_1\mu_2}\hat{\lambda}\\
&>0.
\end{align*}
The proof is complete.
\end{proof}

Now we consider the leftward propagation dynamics of $\{Q_t\}_{t\geq0}$ restricted on $\mathcal C_{[\bm{\alpha}_2,\bm{1}]}$.
Let $\mathcal C_1=\{u\in\mathcal C: 0\leq u\leq 1\}$, where $\mathcal C$ denotes
the set of all bounded and continuous functions from $\Bbb R$ to $\Bbb R$.
Since $\bm{\alpha}_2=(0,1)$, we can easily see that this dynamics is equivalent
to that of the following periodic scalar equation restricted on $\mathcal C_1$:
\begin{equation}\label{sca-2}
u_t=L_1^*u+a_{11}^*(x)u(1-u),
\end{equation}
where $0$ is unstable and $1$ is stable steady states of \eqref{sca-2}.
Since $\lambda_0(d_1,a_1^*,a_{11}^*)>0$, we see that
\eqref{sca-2} admits a leftward spreading speed (also the minimal rightward wave speed, see, e.g., \cite{berestycki2002})
$$c_{1-}^*=\mathop{\inf}\limits_{\mu>0}\frac{\lambda_1(-\mu)}{\mu},$$
where $\lambda_1(\mu)$ is the principal eigenvalue of the following periodic eigenvalue problem
\begin{equation}\label{-}
\begin{cases}
\lambda\phi=d_1\phi^{\prime\prime}-(2d_1\mu+a_1^*)\phi^{\prime}+(d_1\mu^2+a_1^*\mu+a_{11}^*)\phi,\ x\in\Bbb{R},\\
\phi(x+L)=\phi(x),\ x\in\Bbb{R},
\end{cases}
\end{equation}
with $\lambda_1(0)=\lambda_0(d_1,a_1^*,a_{11}^*)>0$. Then we have $c_-^*(\bm{\alpha}_2,\bm{1})\geq c_{1-}^*$.
Similarly, the rightward propagation dynamics of $\{Q_t\}_{t\geq0}$ restricted on $\mathcal C_{[\bm{0},\bm{\alpha}_2]}$
is equivalent to that of the following periodic scalar equation restricted on $\mathcal C_1$:
\begin{equation*}
v_t=L_2^*v-a_{22}^*(x)v(1-v),
\end{equation*}
where $\mathcal C_1=\{u\in\mathcal C: 0\leq u\leq 1\}$, and $\mathcal C$ denotes
the set of all bounded and continuous functions from $\Bbb R$ to $\Bbb R$.
Let $w(t,x)=1-v(t,x)$, then $w$ satisfies the periodic scalar equation
\begin{equation}\label{sca-1}
w_t=L_2^*w+a_{22}^*(x)w(1-w),
\end{equation}
where $0$ is unstable and $1$ is stable steady states of \eqref{sca-1}.
In view of $\lambda_0(d_2,a_2^*,a_{22}^*)>0$,
a similar argument as in the proof of \cite[Theorem 1.1]{berestycki20051}
shows that \eqref{sca-1} admits a rightward spreading speed (also the minimal rightward wave speed)
$$c_{2+}^*=\mathop{\inf}\limits_{\mu>0}\frac{\lambda_2(\mu)}{\mu},$$
where $\lambda_2(\mu)$ is the principal eigenvalue of the following periodic eigenvalue problem
\begin{equation}\label{+}
\begin{cases}
\lambda\phi=d_2\phi^{\prime\prime}-(2d_2\mu+a_2^*)\phi^{\prime}+(d_2\mu^2+a_2^*\mu+a_{22}^*)\phi,\ x\in\Bbb{R},\\
\phi(x+L)=\phi(x),\ x\in\Bbb{R},
\end{cases}
\end{equation}
with $\lambda_2(0)=\lambda_0(d_2,a_2^*,a_{22}^*)>0$. Then we have $c_+^{*}(\bm{0},\bm{\alpha}_2)\geq c_{2+}^*$.
We finally need the following assumption to complete the verification of (A6) in \cite{fang2011}.
\begin{itemize}
\item[(B2)] $c_{1-}^*+c_{2+}^*>0$, where $c_{1-}^*$ and $c_{2+}^*$ are the leftward and rightward spreading speeds of
\eqref{sca-2} and \eqref{sca-1}, respectively.
\end{itemize}

\begin{remark}\rm
In the case that
\begin{itemize}
\item[(i)] either $L_i^*u=(d_i(x)u_x)_x$ with $d_i\in C^{1+\nu}(\Bbb{R})$,
\item[(ii)] or $d_i(\cdot),\ a_{ii}^*(\cdot)$ are even, and $a_i^*(\cdot)$ are odd for $i=1,2$,
\end{itemize}
(A1) together with (A2) guarantees (B2). Indeed, in the either of the two cases (i) and (ii),
we know from \cite[Lemma 5.1]{Yu2017} that $\lambda_i(\mu),\ i=1,2$ are even and convex functions
of $\mu$ in $\Bbb{R}$, which together with the fact
$\lambda_1(0)=\lambda_0(d_1,a_1^*,a_{11}^*)>0$ and $\lambda_2(0)=\lambda_0(d_2,a_2^*,a_{22}^*)>0$ yields that
$\lambda_i(\mu)>\lambda_i(0)>0$ for any $\mu>0$ and $i=1,2$. Therefore, we can easily see
$c_{1-}^*=\mathop{\inf}\limits_{\mu>0}\frac{\lambda_1(-\mu)}{\mu}>0$ and
$c_{2+}^*=\mathop{\inf}\limits_{\mu>0}\frac{\lambda_2(\mu)}{\mu}>0$, which directly leads to (B2).
One should also note from \eqref{-}, \eqref{+} and the representations of $c_{1-}^*$ and $c_{2+}^*$ that,
(B2) could not hold generally if the advection terms $a_1^*$ and $a_2^*$ being very large opposite constants.
\end{remark}

Combining Lemma \ref{coun-1} and the above argument, we easily obtain the following result.
\begin{proposition}\label{count}
Assume (A1)-(A2) and (B1)-(B2). Then for any $\bm{\alpha}\in E\setminus\{\bm{0},\bm{1}\}$, there holds
$c_+^*(\bm{0},\bm{\alpha})+c_-^*(\bm{\alpha},\bm{1})>0$.
\end{proposition}

Now we are ready to state our main results of this section.

\begin{theorem}\label{main}
Assume that (A1)-(A2) and (B1)-(B2). Then there exists some $c\in\Bbb{R}$ such that system \eqref{u1u2-0-1} admits a pulsating front
$\bm{U}(x,x+ct)=(U_1(x,x+ct),U_2(x,x+ct))$ satisfying $\mathop{\lim}\limits_{z\to-\infty}\bm{U}(x,z)=\bm{0}$ and
$\mathop{\lim}\limits_{z\to+\infty}\bm{U}(x,z)=\bm{1}$ uniformly for $x\in\Bbb{R}$.
Furthermore, $\bm{U}(x,z)$ is increasing in $z$ for any $x\in\Bbb{R}$ if $c\neq 0$.
\end{theorem}

\begin{proof}
In view of Proposition \ref{count}, it follows from \cite[Theorem 4.1]{fang2011} that system \eqref{u1u2-0-1} admits a pulsating front
$\bm{U}(x,x+ct)=(U_1(x,x+ct),U_2(x,x+ct))$ connecting $\bm{0}$ to $\bm{1}$, that is,
\begin{equation}\label{limit}
\mathop{\lim}\limits_{z\to-\infty}\bm{U}(x,z)=\bm{0},\quad \mathop{\lim}\limits_{z\to+\infty}\bm{U}(x,z)=\bm{1},
\text{~~uniformly for~~} x\in\Bbb{R}.
\end{equation}
Besides, $\bm{U}(x,z)$ is nondecreasing in $z$.
Next we prove that $\bm{U}(\cdot,z)$ is strictly increasing.
Let $$(u_1(t,x),u_2(t,x)):=(U_1(x,x+ct),U_2(x,x+ct)).$$
For any fixed $\tau>0$, define
$$(r_1(t,x),r_2(t,x))=\left(u_1\left(t+\frac{\tau}{c},x\right)-u_1(t,x),u_2\left(t+\frac{\tau}{c},x\right)-u_2(t,x)\right).$$
Then $(r_1(t,x),r_2(t,x))$ is nonnegative in $(t,x)\in\Bbb{R}\times\Bbb{R}$ satisfying
\begin{equation*}
\begin{cases}
\frac{\partial r_1}{\partial t}
=L_1^*r_1+f_1\left(x,u_1\left(t+\frac{\tau}{c},x\right),u_2\left(t+\frac{\tau}{c},x\right)\right)-f_1(x,u_1(t,x),u_2(t,x)),\\
\frac{\partial r_2}{\partial t}
=L_2^*r_2+f_2\left(x,u_1\left(t+\frac{\tau}{c},x\right),u_2\left(t+\frac{\tau}{c},x\right)\right)-f_2(x,u_1(t,x),u_2(t,x)),
\end{cases}
(t,x)\in\Bbb{R}\times\Bbb{R}.
\end{equation*}
Assume now that there exists some $(\bar{t},\bar{x})\in\Bbb{R}\times\Bbb{R}$
such that either $r_1(\bar{t},\bar{x})=0$ or $r_2(\bar{t},\bar{x})=0$. Since $f_{1,u_2}, f_{2,u_1}\geq 0$, one have
$$
\frac{\partial r_1}{\partial t}\geq L_1^*r_1+r_1\left[\int_0^1{f_{1,u_1}\left(x,su_1\left(t+\frac{\tau}{c},x\right)
+(1-s)u_1(t,x),su_2\left(t+\frac{\tau}{c},x\right)+(1-s)u_2(t,x)\right)}ds\right]
$$
and
$$
\frac{\partial r_2}{\partial t}\geq L_2^*r_2+r_2\left[\int_0^1{f_{2,u_2}\left(x,su_1\left(t+\frac{\tau}{c},x\right)
+(1-s)u_1(t,x),su_2\left(t+\frac{\tau}{c},x\right)+(1-s)u_2(t,x)\right)}ds\right].
$$
It then follows from the (strong) maximum principle that either $r_1\equiv0$ or $r_2\equiv0$ in $(-\infty,\bar{t}]\times\Bbb{R}$,
and then in $\Bbb{R}\times\Bbb{R}$ by uniqueness of the Cauchy problem associated to \eqref{u1u2-0-1}, which together with the fact
$f_{1,u_2}\geq 0$ and $f_{2,u_1}\geq 0$ yields that $(r_1,r_2)\equiv(0,0)$ in $\Bbb{R}\times\Bbb{R}$. In particular,
$$\left(u_1\left(\frac{k\tau}{c},x\right),u_2\left(\frac{k\tau}{c},x\right)\right)=(u_1(0,x),u_2(0,x)),\ \forall x\in\Bbb{R},\ \forall k\in\Bbb{Z}.$$
On the other hand, since for each $x\in\Bbb{R}$,
$$\mathop{\lim}\limits_{k\to-\infty}\left(u_1\left(\frac{k\tau}{c},x\right),u_2\left(\frac{k\tau}{c},x\right)\right)=(0,0),\quad
\mathop{\lim}\limits_{k\to+\infty}\left(u_1\left(\frac{k\tau}{c},x\right),u_2\left(\frac{k\tau}{c},x\right)\right)=(1,1),$$
due to \eqref{limit} and $\tau>0$. One has reached a contradiction. Then $(r_1(t,x),r_2(t,x))>(0,0)$ in $(t,x)\in\Bbb{R}\times\Bbb{R}$, and
therefore $(U_1(x,z+\tau),U_2(x,z+\tau))>(U_1(x,z),U_2(x,z))$ for any $(x,z)\in\Bbb{R}\times\Bbb{R}$ and $\tau>0$. The proof is complete.
\end{proof}

\begin{remark}\rm
The pulsating fronts established in Theorem \ref{main} are leftward propagating. Under assumption that $c_{1+}^*+c_{2-}^*>0$,
where $c_{1+}^*$ and $c_{2-}^*$ are the rightward and leftward spreading speeds of \eqref{sca-2} and \eqref{sca-1}, respectively,
we can similarly obtain a rightward pulsating front $\bm{V}(x,x-ct)=\bm{V}(x,z)$ connecting $\bm{1}$ to $\bm{0}$,
which turned out to be decreasing in $z$ for any $x\in\Bbb{R}$.
\end{remark}

\section{Sub- and supersolutions}

In this section, we first give a comparison lemma to an auxiliary system,
which admits the comparison principle in a larger interval and coincides with system \eqref{u1u2-0-1}
in $[\bm{0},\bm{1}]$, since the sub- and supersolutions constructed later
may be unbounded from below by $\bm{0}$ and above by $\bm{1}$.
Then we construct a pair of sub-super solutions preparing for the next section.

Consider the following auxiliary system
\begin{equation}\label{F1F2}
\begin{cases}
\frac{\partial u_1}{\partial t}=d_1(x)\frac{\partial u_1^2}{\partial x^2}-a_1^*(x)\frac{\partial u_1}{\partial x}+F_1(x,u_1,u_2),\\
\frac{\partial u_2}{\partial t}=d_2(x)\frac{\partial u_2^2}{\partial x^2}-a_2^*(x)\frac{\partial u_2}{\partial x}+F_2(x,u_1,u_2),
\end{cases}
t\in\Bbb{R}^+,\ x\in\Bbb{R},
\end{equation}
where for any $(u_1,u_2)\in[\bm{-1},\bm{2}]$,
\begin{align*}
F_1(x,u_1,u_2)&:=f_1(x,u_1,u_2)+D_1\{u_1\}^-u_2,\\
F_2(x,u_1,u_2)&:=f_2(x,u_1,u_2)+D_2\{1-u_2\}^-(u_1-1),
\end{align*}
$\{w\}^-:=\max\left\{-w,0\right\}$, $D_1=\mathop{\max}\limits_{x\in[0,L]}a_{12}^*(x)$ and  $D_2=\mathop{\max}\limits_{x\in[0,L]}a_{21}^*(x)$.
We first give the definition of sub- and supersolutions of \eqref{F1F2} as following.

\begin{definition}
A pair of continuous functions $\bm{w}(t,x)=(u_1(t,x),u_2(t,x))$ is said to be a supersolution (resp. subsolution)
of \eqref{F1F2} in $(t,x)\in\Bbb R^+\times\Bbb R$,
if $\bm{w}\in C^{1,2}((0,\infty)\times\Bbb{R},\Bbb{R}^2)$ and
\begin{equation*}
\begin{cases}
u_{1,t}-d_1(x)\frac{\partial u_1^2}{\partial x^2}+a_1^*(x)\frac{\partial u_1}{\partial x}-F_1(x,u_1,u_2)\geq 0 ~(resp. \leq0),~~(t,x)\in\Bbb R^+\times\Bbb R,\\
u_{2,t}-d_2(x)\frac{\partial u_2^2}{\partial x^2}+a_2^*(x)\frac{\partial u_2}{\partial x}-F_2(x,u_1,u_2)\geq 0 ~(resp. \leq0),~~(t,x)\in\Bbb R^+\times\Bbb R.
\end{cases}
\end{equation*}
\end{definition}

Similar to \cite[Theorem 2.2 and Corollary 2.3]{wang2012} (see also \cite[Theorem 3.2]{bao2013}),
we have the following comparison lemma.

\begin{lemma}\label{sub-super method}
\rm(i)
Suppose that $\bm{w}^+(t,x)$ and $\bm{w}^-(t,x)$ are super- and subsolutions of \eqref{F1F2} in
$(t,x)\in\Bbb R^+\times\Bbb R$, respectively, $\bm{-1}\leq\bm{w}^\pm(t,x)\leq\bm{2}$.
If $\bm{w}^-(0,x)\leq\bm{w}^+(0,x)$ for any $x\in\Bbb R$,
then $\bm{w}^-(t,x)\leq\bm{w}^+(t,x)$ for any $(t,x)\in\Bbb R^+\times\Bbb R$.

\rm(ii)
For any $\bm{w}_0(\cdot)\in C(\Bbb{R},[\bm{0},\bm{1}])$, if
$\bm{w}^-(0,\cdot)\leq\bm{w}_0(\cdot)\leq\bm{w}^+(0,\cdot)$ for any $x\in\Bbb R$,
then
$\bm{w}^-(t,x)\leq\bm{w}(t,x;\bm{w}_0)\leq\bm{w}^+(t,x)$ and
$\bm{0}\leq\bm{w}(t,x;\bm{w}_0)\leq\bm{1}$ for any $(t,x)\in\Bbb R^+\times\Bbb R$,
where
$\bm{w}(t,x;\bm{w}_0)$ is the unique classical solution of \eqref{F1F2} with $\bm{w}(0,x;\bm{w}_0)=\bm{w}_0$.

\rm(iii)
For any $\bm{w}_0(\cdot)\in C(\Bbb{R},[\bm{0},\bm{1}])$, the unique classical solution
$\bm{w}(t,x;\bm{w}_0)$ of \eqref{F1F2} with $\bm{w}(0,x;\bm{w}_0)=\bm{w}_0$ is also a
classical solution of \eqref{u1u2-0-1}.
\end{lemma}

Now we are ready to construct a pair of appropriate sub- and supersolutions of \eqref{F1F2} using the pulsating front
$\bm{U}(x,x+ct)=(U_1(x,x+ct),U_2(x,x+ct))$.

Let $\chi$ be a smooth function satisfying $\chi(\xi)=0$ for any $\xi\leq-2$,
$\chi(\xi)=1$ for any $\xi\geq2$, $0\leq\chi^\prime\leq1$ and $|\chi^{\prime\prime}|\leq1$.
Define a positive $L$-periodic vector $\bm{p}(x,\xi)=(p_1(x,\xi),p_2(x,\xi))$ with
\begin{align*}
p_1(x,\xi)&=(1-\chi(\xi))\phi_{01}^*(x)+\chi(\xi)\phi_{11}^*(x),~(x,\xi)\in\Bbb R^2,\\
p_2(x,\xi)&=(1-\chi(\xi))\phi_{02}^*(x)+\chi(\xi)\phi_{12}^*(x),~(x,\xi)\in\Bbb R^2,
\end{align*}
where $(\phi_{i1}^*(x),\phi_{i2}^*(x))~(i=0,1)$ are defined in Lemma \ref{0-1-eigen-eigen}.
Furthermore, denote
$$\rho^*=\max_{i,j=0,1}\left\{\mathop{\max}\limits_{x\in[0,L]}\phi^*_{ij}(x)\right\},
\quad\rho_*=\min_{i,j=0,1}\left\{\mathop{\min}\limits_{x\in[0,L]}\phi^*_{ij}(x)\right\}.$$

\begin{lemma}\label{TWS-sub-super}
There exist positive constants $\beta_0,~\delta_0$ and $\sigma_0$ such that for any $z^\pm\in\Bbb R$ and $\delta\in(0,\delta_0]$,
the functions $\bm{u}^\pm(t,x)=(u_1^\pm(t,x),u_2^\pm(t,x))$ defined by
\begin{equation*}
u_i^\pm(t,x)=U_i(x,x+ct+z^\pm\pm\sigma_0\delta(1-e^{-\beta_0 t}))
\pm\delta p_i(x,x+ct+z^\pm\pm\sigma_0\delta(1-e^{-\beta_0 t}))e^{-\beta_0 t},~i=1,2
\end{equation*}
are a pair of super- and subsolutions of \eqref{F1F2} for any $(t,x)\in(0,+\infty)\times\Bbb R$.
\end{lemma}

\begin{proof}
The proof is analogous to that of \cite[Lemma 3.4]{bao2013}, for the completeness of the present paper, we give the details.
Let $\xi(t,x)=x+ct+z^++\sigma_0\delta(1-e^{-\beta_0 t})$. Then
\begin{align*}
&(u_1^+(t,x),u_2^+(t,x))=\left(U_1(x,\xi)+\delta p_1(x,\xi)e^{-\beta_0 t},~U_2(x,\xi)+\delta p_2(x,\xi)e^{-\beta_0 t}\right)>(0,0),\\
&F_1(x,u_1^+,u_2^+)=f_1(x,u_1^+,u_2^+),\quad F_2(x,u_1^+,u_2^+)=f_2(x,u_1^+,u_2^+)+D_2\{1-u_2^+\}^-(u_1^+-1).
\end{align*}
We only prove that $u^+_{2,t}-L_2^*u_2^+-F_2(x,u_1^+,u_2^+)\geq0$, since the others can be proved similarly.
Direct calculation shows that
\begin{align*}
&u^+_{2,t}-d_2u^+_{2,xx}+a_2^*u^+_{2,x}-F_2(x,u_1^+,u_2^+)\\
&=\delta e^{-\beta_0 t}\left[(c+\beta_0\sigma_0\delta e^{-\beta_0 t})p_{2,\xi}-\beta_0p_2
-d_2(p_{2,xx}+2p_{2,x\xi}+p_{2,\xi\xi})+a_2^*(p_{2,x}+p_{2,\xi})\right]\\
&\quad+\beta_0\sigma_0\delta U_{2,z}e^{-\beta_0 t}+f_2(x,U_1,U_2)-f_2(x,u_1^+,u_2^+)\\
&\quad-D_2\{1-u^+_2\}^-(u^+_1-1)\\
&=\delta e^{-\beta_0 t}\left[(c+\beta_0\sigma_0\delta e^{-\beta_0 t})p_{2,\xi}-\beta_0p_2
-d_2(p_{2,xx}+2p_{2,x\xi}+p_{2,\xi\xi})+a_2^*(p_{2,x}+p_{2,\xi})\right]\\
&\quad+\beta_0\sigma_0\delta U_{2,z}e^{-\beta_0 t}-
\left[\int_0^1{f_{2,u_1}\left(x,\widetilde{U}_1(x,\xi;\theta),\widetilde{U}_2(x,\xi;\theta)\right)d\theta}\right]
\delta p_1(x,\xi)e^{-\beta_0 t}\\
&\quad-\left[\int_0^1{f_{2,u_2}\left(x,\widetilde{U}_1(x,\xi;\theta),\widetilde{U}_2(x,\xi;\theta)\right)d\theta}\right]
\delta p_2(x,\xi)e^{-\beta_0 t}\\
&\quad-D_2\{1-u^+_2\}^-(u^+_1-1),
\end{align*}
where $\widetilde{U}_i(x,\xi;\theta):=U_i(x,\xi)+\theta\delta p_i(x,\xi)e^{-\beta_0t}$, $\theta\in(0,1)$, $~i=1,2$.

Denote
\begin{align*}
\Gamma_0(x,\xi;\theta)&=\mathop{\sum}\limits_{i,j=1,2}\left|f_{i,u_j}(x,\widetilde{U}_1,\widetilde{U}_2)-f_{i,u_j}(x,0,0)\right|,\\
\Gamma_1(x,\xi;\theta)&=\mathop{\sum}\limits_{i,j=1,2}\left|f_{i,u_j}(x,\widetilde{U}_1,\widetilde{U}_2)-f_{i,u_j}(x,1,1)\right|.
\end{align*}
Noting that
$$\mathop{\lim}\limits_{\xi\to-\infty}(U_1(x,\xi),U_2(x,\xi))=(0,0),\
\mathop{\lim}\limits_{\xi\to+\infty}(U_1(x,\xi),U_2(x,\xi))=(1,1)\text{~~uniformly for~~}x\in\Bbb R.$$
Then there exist some $\hat{\xi}>2$ large enough and $\delta_1\in(0,1)$ small enough such that
for any $\theta\in(0,1)$ and $\delta\in(0,\delta_1)$, it follows
\begin{align*}
\mathop{\sup}\limits_{(x,\xi)\in\Bbb R\times(-\infty,-\hat{\xi}]}
\left|\Gamma_0(x,\xi;\theta)\right|&\leq\frac{|\mu_0^-|
\min\left\{\mathop{\min}\limits_{x\in[0,L]}\phi^*_{01},\mathop{\min}\limits_{x\in[0,L]}\phi^*_{02}\right\}}
{2\mathop{\max}\limits_{x\in[0,L]}(\phi^*_{01}+\phi^*_{02})},\\
\mathop{\sup}\limits_{(x,\xi)\in\Bbb R\times[\hat{\xi},+\infty)}
\left|\Gamma_1(x,\xi;\theta)\right|&\leq\frac{|\mu_1^-|
\min\left\{\mathop{\min}\limits_{x\in[0,L]}\phi^*_{11},\mathop{\min}\limits_{x\in[0,L]}\phi^*_{12}\right\}}
{2\mathop{\max}\limits_{x\in[0,L]}(\phi^*_{11}+\phi^*_{12})}.
\end{align*}
Denote
\begin{align*}
d&=\max\left\{\mathop{\max}\limits_{x\in[0,L]}d_1(x), \mathop{\max}\limits_{x\in[0,L]}d_2(x)\right\},\quad
a^*=\max\left\{\mathop{\max}\limits_{x\in[0,L]}|a_1^*(x)|, \mathop{\max}\limits_{x\in[0,L]}|a_2^*(x)|\right\},\\
D&=\max\{D_1, D_2\},\quad
C_1=\max\left\{|p_1|_{C^{2,2}([0,L]\times\Bbb R)}, |p_2|_{C^{2,2}([0,L]\times\Bbb R)}\right\},\\
C_2&=\mathop{\sup}\limits_{(u_1,u_2)\in[\bm{-1},\bm{2}],~x\in\Bbb R}
\left\{\sum_{i,j=1,2}|f_{i,u_j}(x,u_1,u_2)|\right\},\
C_3=\min_{i=1,2}\left\{\mathop{\min}\limits_{x\in[0.L],~\xi\in[-\hat{\xi},\hat{\xi}]}
\frac{\partial U_i}{\partial z}\right\}.
\end{align*}
For any fixed $\beta_0\in\left(0,\min\left\{\frac{|\mu_0^-|}{4},\frac{|\mu_1^-|}{4}\right\}\right)$, let
$$\delta_0=\min\left\{\delta_1, \frac{|\mu_0^-|}{4D\rho^*}, \frac{|\mu_1^-|}{4D\rho^*}, \frac{C_3}{2C_1}\right\},\quad
\sigma_0\geq\frac{2C_1[c+\beta_0+4d+2a^*+\delta_0C_1D+C_2]}{\beta_0C_3}.$$
Consider the following three cases.
\begin{description}
\item[Case~1.]  $\xi(t,x)\geq\hat{\xi}$.
Then $p_1(x,\xi)=\phi_{11}^*(x)$ and $p_2(x,\xi)=\phi_{12}^*(x)$, and
\begin{align*}
&\mathcal{F}_2(x,u_1^+,u_2^+)\\
&\geq\delta e^{-\beta_0 t}\left[-\beta_0\phi_{12}^*
-d_2(\phi_{12}^*)^{\prime\prime}+a_2^*(\phi_{12}^*)^{\prime}\right]\\
&\quad-\left[\int_0^1{f_{2,u_1}\left(x,\widetilde{U}_1(x,\xi;\theta),\widetilde{U}_2(x,\xi;\theta)\right)d\theta}\right]
\delta e^{-\beta_0 t}\cdot\phi_{11}^*\\
&\quad-\left[\int_0^1{f_{2,u_2}\left(x,\widetilde{U}_1(x,\xi;\theta),\widetilde{U}_2(x,\xi;\theta)\right)d\theta}\right]
\delta e^{-\beta_0 t}\cdot\phi_{12}^*\\
&\quad-D_2\{1-u^+_2\}^-(u^+_1-1)\\
&=\delta e^{-\beta_0 t}\left\{\int_0^1\left(f_{2,u_1}(x,1,1)-
f_{2,u_1}(x,\widetilde{U}_1(x,\xi;\theta),\widetilde{U}_2(x,\xi;\theta))\right)d\theta\cdot\phi_{11}^*\right .\\
&\quad+\int_0^1\left(f_{2,u_2}(x,1,1)
\left .-f_{2,u_2}(x,\widetilde{U}_1(x,\xi;\theta),\widetilde{U}_2(x,\xi;\theta))\right)d\theta\cdot\phi_{12}^*
-(\mu_1^-+\beta_0)\phi_{12}^*\right\}\\
&\quad-D_2\{1-u^+_2\}^-(u^+_1-1)\\
&\geq\delta e^{-\beta_0 t}\phi_{12}^*
\left[-\frac{|\mu_1^-|}{2}-\mu_1^--\beta_0-D_2\delta\phi_{11}^*\right]\\
&\geq0.
\end{align*}
\item[Case~2.]  $|\xi(t,x)|\leq\hat{\xi}$.
Then for any $\delta\in(0,\delta_0]$,
\begin{align*}
&\mathcal{F}_2(x,u_1^+,u_2^+)\\
&\geq-\delta e^{-\beta_0 t}\left[(c+\beta_0\sigma_0\delta)C_1+\beta_0C_1
+4dC_1+2a^*C_1\right]\\
&\quad+\beta_0\sigma_0\delta C_3e^{-\beta_0 t}
-C_2\delta C_1e^{-\beta_0 t}
-D\delta C_1e^{-\beta_0 t}\delta C_1e^{-\beta_0 t}\\
&\geq\delta e^{-\beta_0 t}
\left[\beta_0\sigma_0( C_3-\delta C_1)-C_1(c+\beta_0+4d+2a^*+C_2+D\delta_0C_1)\right]\\
&\geq0.
\end{align*}
\item[Case~3.]  $\xi(t,x)\leq-\hat{\xi}$.
Then $p_1(x,\xi)=\phi_{01}^*(x)$ and $p_2(x,\xi)=\phi_{02}^*(x)$. Similar to the proof of
$\bm{Case~1}$, one can prove that $\mathcal{F}_2(x,u_1^+,u_2^+)\geq0$.
\end{description}
Consequently, $\mathcal{F}_2(x,u_1^+,u_2^+)\geq0$ for any $(t,x)\in(0,+\infty)\times\Bbb R$.
With analogous arguments as above, we complete the proof.
\end{proof}

\section{Global stability and uniqueness}

This section is devoted to the study of global stability and uniqueness of the pulsating
front $\bm{U}(x,x+ct)=(U_1(x,x+ct),U_2(x,x+ct))$ with nonzero wave speed.
Without loss of generality, we assume that the wave speed $c>0$, the case $c<0$ can be discussed similarly.
By using the convergence theorem for monotone semiflows (see \cite[Theorem 2.2.4]{zhao2003}),
we show that the solution of the Cauchy problem associated to \eqref{u1u2-0-1} with some proper initial values
converges to a translation of the pulsating front as $t\to+\infty$.
The general strategy of the proof is to trap the solution of a Cauchy problem related to \eqref{u1u2-0-1} between suitable sub- and supersolutions
established in Section 3 which are close to some shifts of the periodic traveling wave $\bm{U}(x,x+ct)=(U_1(x,x+ct),U_2(x,x+ct))$.
The uniqueness result can then be viewed as a consequence of the global stability.

Consider the moving coordinates
$$(t,z)=(t,x+ct).$$
By rescaling system \eqref{u1u2-0-1} into the moving coordinates $(t,z)$, that is, let
$$\bm{v}(t,z)=(v_1(t,z),v_2(t,z))=(v_1(t,x+ct),v_2(t,x+ct)),$$
we transform \eqref{u1u2-0-1} into the following system
\begin{equation}\label{TWS-t-z}
\begin{cases}
\frac{\partial v_1(t,z)}{\partial t}=d_1(z-ct)\frac{\partial^2 v_1}{\partial z^2}
-(a_1^*(z-ct)+c)\frac{\partial v_1}{\partial z}+f_1(z-ct,v_1,v_2),\\
\frac{\partial v_2(t,z)}{\partial t}=d_2(z-ct)\frac{\partial^2 v_2}{\partial z^2}
-(a_2^*(z-ct)+c)\frac{\partial v_2}{\partial z}+f_2(z-ct,v_1,v_2),
\end{cases}
t>0,~z\in\Bbb R.
\end{equation}
Then \eqref{TWS-t-z} is a time periodic system with period $T=\frac{L}{c}>0$.
Define $Q_t(\bm{v}_0):=\bm{v}(t,\cdot;\bm{v}_0)$ for any $\bm{v}_0\in\mathcal C_{\bm{1}}$ and $t\geq0$, where
$\bm{v}(t,\cdot;\bm{v}_0)$ is the unique solution of \eqref{TWS-t-z} with initial value $\bm{v}_0$.
Let $P :\mathcal C_{\bm{1}}\to \mathcal C_{\bm{1}}$ be the Poincar$\acute{e}$ map associated with the
periodic semiflow $Q_t(\cdot)$, that is, $P(\bm{v}_0)=Q_T(\bm{v}_0)=\bm{v}(T,\cdot;\bm{v}_0)$.
For any $\tau\in\Bbb{R}$, since the pulsating front $\bm{U}(x,z+\tau)=\bm{U}(x,x+ct+\tau)$ is a solution of system \eqref{u1u2-0-1},
it follows that $$\bm{V}^\tau(t,z):=\bm{U}(z-ct,z+\tau)$$ is a $T$-periodic solution of \eqref{TWS-t-z} satisfying $\bm{V}^{\tau_1}(0,\cdot)>\bm{V}^{\tau_2}(0,\cdot)$ for all $\tau_1>\tau_2$ by the strictly monotonicity of $\bm{U}(\cdot,z)$ with respect to $z$.
Furthermore, it can be seen that $\bm{0}$, $\bm{1}$ and $\bm{V}^\tau(0,\cdot)$ are fixed points of $P$ in $\mathcal C_{\bm{1}}$, with
the set $\{\bm{V}^\tau(0,\cdot)|\tau\in\Bbb{R}\}$ totally ordered in $\mathcal C_{\bm{1}}$.
Next we shall apply the following convergence lemma to the Poincar$\acute{e}$ map $P$ and
its fixed points $\{\bm{V}^\tau(0,\cdot)|\tau\in\Bbb{R}\}$.

\begin{lemma}[\rm{\cite[Theorem 2.2.4]{zhao2003}}]\label{convergence}
Let $M$ be a closed convex subset of an order Banach space $\mathcal{X}$, and $\Phi :M\to M$ be a monotone semiflow.
Assume that there exists a monotone homeomorphism $\zeta$ from $[0,1]$ onto a subset of $M$ such that
\begin{itemize}
\item[(1)] For each $s\in[0,1]$, $\zeta(s)$ is a stable equilibrium for $\Phi :M\to M$.
\item[(2)] Each orbit of $\Phi$ in $[\zeta(0),\zeta(1)]$ is precompact.
\item[(3)] One of the following two properties holds:
\begin{itemize}
\item[(a)] if $\zeta(s_0)<_\mathcal{X}\omega(\phi)$ for some $s_0\in[0,1)$ and $\phi\in\mathcal{X}_{[\zeta(0),\zeta(1)]}$,
then there exists $s_1\in(s_0,1)$ such that $\zeta(s_1)\leq_\mathcal{X}\omega(\phi)$,
here $\omega(\phi)$ is the $\omega$-limit set of $\phi$ in $\Phi$;
\item[(b)] if $\omega(\phi)<_\mathcal{X}\zeta(r_1)$ for some $r_1\in(0,1]$ and $\phi\in\mathcal{X}_{[\zeta(0),\zeta(1)]}$,
then there exists $r_0\in(0,r_1)$ such that $\omega(\phi)\leq_\mathcal{X}\zeta(r_0)$.
\end{itemize}
\end{itemize}
Then for any precompact orbit $\gamma^+(\phi_0)$ of $\Phi$ in $M$
with $\omega(\phi_0)\cap[\zeta(0),\zeta(1)]_{\mathcal{X}}\neq\emptyset$, there exists
$s_*\in[0,1]$ such that $\omega(\phi_0)=\zeta(s_*)$.
\end{lemma}

Next we provide two lemmas for the preparation of using Lemma \ref{convergence}.

\begin{lemma}\label{G1}
Assume that $\bm{v}_0\in\mathcal C_{\bm{1}}$ satisfies
\begin{equation}\label{initial}
\mathop{\lim\sup}\limits_{z\to-\infty}\bm{v}_0(z)\ll\delta_1\rho_*\bm{1},\quad
\mathop{\lim\inf}\limits_{z\to+\infty}\bm{v}_0(z)\gg\bm{1}-\delta_1\rho_*\bm{1},
\end{equation}
where $\delta_1=\min\left\{\delta_0, \frac{1}{\rho_*}\right\}$ with $\delta_0$ defined in Lemma \ref{TWS-sub-super}.
Then for any $\varepsilon>0$, there exist some
$\hat{\tau}=\hat{\tau}(\varepsilon,\bm{v}_0)>0$ and an integer $\hat{k}=\hat{k}(\varepsilon,\bm{v}_0)>0$ such that
$$\bm{V}^{-\hat{\tau}}(0,z)-\varepsilon\bm{1}\leq\bm{v}(\hat{k}T,z;\bm{v}_0)
\leq\bm{V}^{\hat{\tau}}(0,z)+\varepsilon\bm{1},\ \ \forall z\in\Bbb{R}.$$
\end{lemma}

\begin{proof}
Let $\delta^{\prime}\in(0,\delta_1)$ be such that
\begin{equation}\label{ini}
\mathop{\lim\sup}\limits_{z\to-\infty}\bm{v}_0(z)<\delta^{\prime}\rho_*\bm{1},\quad
\mathop{\lim\inf}\limits_{z\to+\infty}\bm{v}_0(z)>\bm{1}-\delta^{\prime}\rho_*\bm{1}.
\end{equation}
Define a pair of $\tilde{\bm{v}}^{\pm}(t,z)=(\tilde{v}_1^{\pm}(t,z),\tilde{v}_2^{\pm}(t,z))$ as
$$\tilde{v}_i^\pm(t,z)=U_i(z-ct,z+z^\pm\pm\sigma_0\delta(1-e^{-\beta_0 t}))
\pm\delta p_i(z-ct,z+z^\pm\pm\sigma_0\delta(1-e^{-\beta_0 t}))e^{-\beta_0 t},\ i=1,2.$$
Then Lemma \ref{TWS-sub-super} implies that $\tilde{\bm{v}}^{\pm}(t,z)$ are sub- and supersolutions of
\begin{equation*}
\begin{cases}
\frac{\partial v_1(t,z)}{\partial t}=d_1(z-ct)\frac{\partial^2 v_1}{\partial z^2}
-(a_1^*(z-ct)+c)\frac{\partial v_1}{\partial z}+F_1(z-ct,v_1,v_2),\\
\frac{\partial v_2(t,z)}{\partial t}=d_2(z-ct)\frac{\partial^2 v_2}{\partial z^2}
-(a_2^*(z-ct)+c)\frac{\partial v_2}{\partial z}+F_2(z-ct,v_1,v_2),
\end{cases}
t>0,~z\in\Bbb R,
\end{equation*}
By \eqref{ini}, there exist $z^+>0$ and $-z^->0$ large enough such that for $\delta\in(\delta^{\prime},\delta_1)$,
$$\bm{U}(z,z+z^-)-\delta\bm{p}(z,z+z^-)\leq\bm{v}_0(z)\leq\bm{U}(z,z+z^+)+\delta\bm{p}(z,z+z^+),\ \ \forall z\in\Bbb{R}.$$
It then follows from the comparison principle that
$$\tilde{\bm{v}}^-(t,z)\leq\bm{v}(t,z;\bm{v}_0)\leq\tilde{\bm{v}}^+(t,z),\ \ \forall t>0,\ z\in\Bbb{R},$$
that is, for any $t>0$, we have
\begin{align*}
&\bm{U}(z-ct,z+z^--\sigma_0\delta(1-e^{-\beta_0 t}))-\delta\bm{p}(z-ct,z+z^--\sigma_0\delta(1-e^{-\beta_0 t}))e^{-\beta_0 t}\\
&\leq\bm{v}(t,z;\bm{v}_0)\\
&\leq\bm{U}(z-ct,z+z^++\sigma_0\delta(1-e^{-\beta_0 t}))+\delta\bm{p}(z-ct,z+z^++\sigma_0\delta(1-e^{-\beta_0 t}))e^{-\beta_0 t}.
\end{align*}
Now for any $\varepsilon>0$, letting $t=\hat{k}T>0$ in the above inequalities,
where $\hat{k}>0$ is a large enough integer satisfying $\delta\rho^*e^{-\beta_0\hat{k}T}\leq\varepsilon$.
It follows that
\begin{align*}
&\bm{U}(z,z+z^--\sigma_0\delta(1-e^{-\beta_0\hat{k}T}))
-\delta\bm{p}(z,z+z^--\sigma_0\delta(1-e^{-\beta_0\hat{k}T}))e^{-\beta_0\hat{k}T}\\
&\leq\bm{v}(\hat{k}T,z;\bm{v}_0)\\
&\leq\bm{U}(z,z+z^++\sigma_0\delta(1-e^{-\beta_0\hat{k}T}))
+\delta\bm{p}(z,z+z^++\sigma_0\delta(1-e^{-\beta_0\hat{k}T}))e^{-\beta_0\hat{k}T}.
\end{align*}
Let $\hat{\tau}:=\max\{z^+,-z^-\}+\sigma_0\delta>0$, we further obtain
$$\bm{U}(z,z-\hat{\tau})-\varepsilon\bm{1}\leq\bm{v}(\hat{k}T,z;\bm{v}_0)\leq\bm{U}(z,z+\hat{\tau})+\varepsilon\bm{1},
\ \ \forall z\in\Bbb{R},$$
that is,
$$\bm{V}^{-\hat{\tau}}(0,z)-\varepsilon\bm{1}\leq\bm{v}(\hat{k}T,z;\bm{v}_0)
\leq\bm{V}^{\hat{\tau}}(0,z)+\varepsilon\bm{1},\ \ \forall z\in\Bbb{R}.$$
The proof is complete.
\end{proof}

\begin{lemma}\label{lyapunov}
For any $\tau\in\Bbb{R}$, the wave profile $\bm{V}^\tau(0,\cdot)$ is lyapunov stable for system \eqref{TWS-t-z},
that is, for any $\epsilon>0$, there exists some constant $\rho>0$ such that for any
$\bm{v}_0\in\mathcal C_{\bm{1}}$ satisfying $\lVert\bm{v}_0(\cdot)-\bm{V}^\tau(0,\cdot)\rVert_{\mathcal C}\leq\rho$,
there is $$\lVert\bm{v}(t,\cdot;\bm{v}_0)-\bm{V}^\tau(t,\cdot)\rVert_{\mathcal C}\leq\epsilon,\ \ \forall t>0.$$
\end{lemma}

\begin{proof}
For any $\epsilon>0$, let
$$\delta\in\left(0,\min\left\{\frac{\epsilon}{\rho^*+\sigma_0 K_0},\delta_0\right\}\right),\quad\rho:=\delta\rho_*,$$
where $\sigma_0$ and $\delta_0$ are defined in Lemma \ref{TWS-sub-super},
and $K_0:=\mathop{\max}\limits_{x\in[0,L]}\lVert\bm{U}_z(x,\cdot)\rVert_{\mathcal C}$.
Then for any $\bm{v}_0\in\mathcal C_{\bm{1}}$ with $\lVert\bm{v}_0(\cdot)-\bm{V}^\tau(0,\cdot)\rVert_{\mathcal C}\leq\rho$,
we have
$$\bm{U}(z,z+\tau)-\delta\bm{p}(z,z+\tau)\leq\bm{v}_0(z)\leq\bm{U}(z,z+\tau)+\delta\bm{p}(z,z+\tau),\ \ \forall z\in\Bbb{R}.$$
The comparison principle then yields that
$$\tilde{\bm{w}}^-(t,z)\leq\bm{v}(t,z;\bm{v}_0)\leq\tilde{\bm{w}}^+(t,z),\ \ \forall t\geq0,\ z\in\Bbb{R},$$
where $\tilde{\bm{w}}^\pm(t,z)$ are defined as in Lemma \ref{TWS-sub-super} with $z^\pm=\tau$.
Therefore,
\begin{align*}
&\mathop{\sup}\limits_{t\geq0}\lVert\bm{v}(t,z;\bm{v}_0)-\bm{U}(z-ct,z+\tau)\rVert_{\mathcal C}\\
&\leq\mathop{\sup}\limits_{t\geq0}\lVert\tilde{\bm{w}}^{\pm}(t,z)-\bm{U}(z-ct,z+\tau)\rVert_{\mathcal C}\\
&\leq\mathop{\sup}\limits_{t\geq0}\lVert\tilde{\bm{w}}^{\pm}(t,z)-\bm{U}(z-ct,z+\tau\pm\sigma_0\delta(1-e^{-\beta_0 t}))\rVert_{\mathcal C}\\
&\quad+\mathop{\sup}\limits_{t\geq0}\lVert\bm{U}(z-ct,z+\tau\pm\sigma_0\delta(1-e^{-\beta_0 t}))-\bm{U}(z-ct,z+\tau)\rVert_{\mathcal C}\\
&\leq\delta\rho^*+\sigma_0\delta K_0\\
&\leq\epsilon,
\end{align*}
that is, $\lVert\bm{v}(t,\cdot;\bm{v}_0)-\bm{V}^\tau(t,\cdot)\rVert_{\mathcal C}\leq\epsilon$ for any $t>0$.
The proof is complete.
\end{proof}

Now we are ready to prove the globally stability of the pulsating front of \eqref{u1u2-0-1}.
Consider the following periodic initial problem associated with system \eqref{u1u2-0-1}
\begin{equation}\label{u1u2-cauchy}
\begin{cases}
\frac{\partial u_1}{\partial t}=L_1^*u_1+u_1\left[a_{11}^*(x)(1-u_1)-a_{12}^*(x)(1-u_2)\right],
\quad t>0,~x\in\Bbb R,\\
\frac{\partial u_2}{\partial t}=L_2^*u_2+(1-u_2)\left[a_{21}^*(x)u_1-a_{22}^*(x)u_2\right],
\quad t>0,~x\in\Bbb R,\\
\bm{u}(0,x)=\bm{u}_0(x),\quad x\in\Bbb R.
\end{cases}
\end{equation}

\begin{theorem}\label{GS}
Let $\bm{U}(x,x+ct)=(U_1(x,x+ct),U_2(x,x+ct))$ be the pulsating front of system \eqref{u1u2-0-1} connecting $\bm{0}$ to $\bm{1}$
with $c\neq0$,
and let $\bm{u}(t,x;\bm{u}_0)$ be the solution of \eqref{u1u2-cauchy} with the initial value $\bm{u}_0\in\mathcal C_{\bm{1}}$.
Then for any $\bm{u}_0$ satisfying \eqref{initial}, there exists $\tilde{\tau}\in\Bbb{R}$ such that
$$\mathop{\lim}\limits_{t\to+\infty}
\lVert\bm{v}(t,\cdot;\bm{u}_0)-\bm{V}^{c\tilde{\tau}}(t,\cdot)\rVert_{\mathcal C}=0,$$
that is,
$$\mathop{\lim}\limits_{t\to+\infty}
\lVert\bm{u}(t,\cdot;\bm{u}_0)-\bm{U}(x,x+ct+c\tilde{\tau})\rVert_{\mathcal C}=0.$$
\end{theorem}

\begin{proof}
Recall that $P :\mathcal C_{\bm{1}}\to \mathcal C_{\bm{1}}$ is the Poincar$\acute{e}$ map associated with the periodic
semiflow $Q_t(\cdot)$ generated by \eqref{TWS-t-z}, that is, $P(\bm{u}_0)=Q_T(\bm{u}_0)=\bm{v}(T,\cdot;\bm{u}_0)$,
where $\bm{v}(t,\cdot;\bm{u}_0)$ is the unique solution of \eqref{TWS-t-z} with initial value $\bm{u}_0$.
Then $P^n(\bm{u}_0)=\bm{v}(nT,\cdot;\bm{u}_0)$ for any $n\geq0$.
Since $\bm{u}_0\in\mathcal C_{\bm{1}}$ satisfies \eqref{initial},
Lemma \ref{G1} then yields that for any $\delta\in(0,\delta_0)$, there exist some
$\hat{\tau}>0$ and an integer $\hat{k}>0$ such that
$$\bm{V}^{-\hat{\tau}}(0,z)-\delta\rho_*\bm{1}\leq\bm{v}(\hat{k}T,z;\bm{u}_0)=P^{\hat{k}}(\bm{u}_0)(z)
\leq\bm{V}^{\hat{\tau}}(0,z)+\delta\rho_*\bm{1},\ \ \forall z\in\Bbb{R}.$$
By the comparison principle, for any $t\geq0$ and $z\in\Bbb{R}$, we have
\begin{equation}\label{T}
\begin{aligned}
&\bm{U}(z-ct,z-\hat{\tau}-\sigma_0\delta(1-e^{-\beta_0 t}))-\delta\bm{p}(z-ct,z-\hat{\tau}-\sigma_0\delta(1-e^{-\beta_0 t}))e^{-\beta_0 t}\\
&\leq\bm{v}(t+\hat{k}T,z;\bm{u}_0)\\
&\leq\bm{U}(z-ct,z+\hat{\tau}+\sigma_0\delta(1-e^{-\beta_0 t}))+\delta\bm{p}(z-ct,z+\hat{\tau}+\sigma_0\delta(1-e^{-\beta_0 t}))e^{-\beta_0 t}.
\end{aligned}
\end{equation}
Noting that the sequence $\{P^n(\bm{u}_0)\}_{n\geq1}=\{\bm{v}(nT,\cdot;\bm{u}_0)\}_{n\geq1}$ is bounded in
$C^1(\Bbb{R},\Bbb{R}^2)$ for any $\bm{u}_0\in\mathcal C_{\bm{1}}$ by the standard parabolic estimates,
and $\mathop{\lim}\limits_{z\to-\infty}\bm{V}^{\tau}(t,z)=\bm{0}$,
$\mathop{\lim}\limits_{z\to+\infty}\bm{V}^{\tau}(t,z)=\bm{1}$ uniformly in $t\in\Bbb{R}$ and for any $\tau\in\Bbb{R}$,
it then follows from \eqref{T} that the forward orbit $\gamma^+(\bm{u}_0):=\{P^n(\bm{u}_0)|n\in\Bbb{N}\}$ is precompact
in $\mathcal C$ and its $\omega$-limit set $\omega(\bm{u}_0)$ is thus nonempty, compact and invariant for $P$.
Denote $\tau_0:=\hat{\tau}+\sigma_0\delta$ and let $t=nT$ in \eqref{T}, we see
$$\omega(\bm{u}_0)\subset H:=\left\{\bm{\phi}\in\mathcal C:\bm{V}^{-\tau_0}(0,\cdot)\leq\bm{\phi}(\cdot)\leq\bm{V}^{\tau_0}(0,\cdot)\right\}.$$
For any $s\in[0,1]$, define $\bm{\zeta}(s)=\bm{V}^{(2s-1)\tau_0}(0,\cdot)$, then
$\bm{\zeta}(s)$ is a monotone homeomorphism from $[0,1]$ onto a subset of $\mathcal C$ and $H=\mathcal C_{[\bm{\zeta}(0),\bm{\zeta}(1)]}$,
and we see from Lemma \ref{lyapunov} that each $\bm{\zeta}(s)$ is a stable fined point of $P$ for $s\in[0,1]$.
Furthermore, $\gamma^+(\bm{\phi})\subset H$ is precompact in $\mathcal C$ for any $\bm{\phi}\in H$.

Now we verify the last condition in Lemma \ref{convergence}.
Suppose that $\bm{\zeta}(s_0)<_{\mathcal C}\omega(\bm{\phi}_0)$ for some $s_0\in[0,1)$ and $\bm{\phi}_0\in H$, that is,
$\bm{V}^{(2s_0-1)\tau_0}(0,\cdot)\leq_{\mathcal C}\bm{\phi}(\cdot)$ and $\bm{V}^{(2s_0-1)\tau_0}(0,\cdot)\neq\bm{\phi}(\cdot)$ for any
$\bm{\phi}\in\omega(\bm{\phi}_0)$. The strong maximum principle then yields that
$\bm{V}^{(2s_0-1)\tau_0}(t,z)<\bm{v}(t,z;\bm{\phi})$ for $(t,z)\in(0,\infty)\times\Bbb{R}$, and particularly,
$$\bm{V}^{(2s_0-1)\tau_0}(0,z)=\bm{V}^{(2s_0-1)\tau_0}(T,z)<P(\bm{\phi})(z)=\bm{\phi}(z),\quad\forall z\in\Bbb{R},\ \forall\bm{\phi}\in\omega(\bm{\phi}_0).$$
Noting that $\bm{V}^{\tau}(0,z)=\bm{U}(z,z+\tau)$ and
$\mathop{\lim}\limits_{z\to\pm\infty}\frac{\partial\bm{U}(x,z)}{\partial z}=\bm{0}$ uniformly for $x\in\Bbb{R}$,
there exists some $z_0>0$ large enough such that
$$0<\bar{\theta}:=\mathop{\sup}\limits_{s,s^{\prime}\in[0,3/2],s\neq s^{\prime},|z|\geq z_0}
\frac{|\bm{V}^{(2s-1)\tau_0}(0,z)-\bm{V}^{(2s^{\prime}-1)\tau_0}(0,z)|}{|s-s^{\prime}|}
\leq\frac{2\rho_*}{5}\mathop{\min}\left\{\frac{\tau_0}{\sigma_0},\delta_0\right\},$$
where $\sigma_0$ and $\delta_0$ are defined in Lemma \ref{TWS-sub-super}.
Since $\omega(\bm{\phi}_0)$ is compact, there exists some $s_1\in(s_0,1)$ such that
$$\bm{V}^{(3s_1-s_0-1)\tau_0}(0,z)\leq\bm{\phi}(z),\quad\forall z\in[-z_0,z_0],\ \bm{\phi}\in\omega(\bm{\phi}_0).$$
For any $\bm{\phi}\in\omega(\bm{\phi}_0)$, let $n_k\to+\infty$ be the sequence such that
$P^{n_k}(\bm{\phi}_0)\to\bm{\phi}$ as $k\to+\infty$. Then there exists an integer $n_{k_0}$ such that
$\lVert P^{n_{k_0}}(\bm{\phi}_0)-\bm{\phi}\rVert_{\mathcal C}\leq\bar{\theta}(s_1-s_0)$. Therefore,
\begin{align*}
&P^{n_{k_0}}(\bm{\phi}_0)(z)-\bm{V}^{(3s_1-s_0-1)\tau_0}(0,z)\\
&=P^{n_{k_0}}(\bm{\phi}_0)(z)-\bm{\phi}(z)+\bm{\phi}(z)-\bm{V}^{(3s_1-s_0-1)\tau_0}(0,z)\\
&\geq-\lVert P^{n_{k_0}}(\bm{\phi}_0)-\bm{\phi}\rVert_{\mathcal C}\bm{1}
-\mathop{\sup}\limits_{|z|\geq z_0}|\bm{V}^{(2s_0-1)\tau_0}(0,z)-\bm{V}^{(3s_1-s_0-1)\tau_0}(0,z)|\bm{1}\\
&\geq-\bar{\theta}(s_1-s_0)\bm{1}-\frac{3\bar{\theta}(s_1-s_0)}{2}\bm{1}=-\frac{5\bar{\theta}(s_1-s_0)}{2}\bm{1}\\
&\geq-\delta_1\bm{p}\left(z,z+(3s_1-s_0-1)\tau_0\right),\quad\forall z\in\Bbb{R},
\end{align*}
where $\delta_1:=\frac{5\bar{\theta}(s_1-s_0)}{2\rho_*}<\delta_0$, the first inequality followed by
$\bm{\phi}(z)-\bm{V}^{(3s_1-s_0-1)\tau_0}(0,z)\geq\bm{0}$ for all $z\in[-z_0,z_0]$ and
$\bm{\phi}(z)>\bm{V}^{(2s_0-1)\tau_0}(0,z)$ for all $|z|\geq z_0$ and $\bm{\phi}\in\omega(\bm{\phi}_0)$.
Lemma \ref{TWS-sub-super} then implies that
\begin{align*}
\bm{v}(t,z;P^{n_{k_0}}(\bm{\phi}_0))&\geq\bm{U}(z-ct,z+(3s_1-s_0-1)\tau_0-\sigma_0\delta_1(1-e^{-\beta_0 t}))\\
&\quad-\delta_1\bm{p}(z-ct,z+(3s_1-s_0-1)\tau_0-\sigma_0\delta_1(1-e^{-\beta_0 t}))e^{-\beta_0 t}
\end{align*}
for any $t>0$ and $z\in\Bbb{R}$.
Letting $t=(n_k-n_{k_0})T$ in the above inequality, we have
$$\bm{v}((n_k-n_{k_0})T,\cdot;P^{n_{k_0}}(\bm{\phi}_0))=P^{n_k-n_{k_0}}(P^{n_{k_0}}(\bm{\phi}_0))(\cdot)
=P^{n_k}(\bm{\phi}_0)(\cdot)\to\bm{\phi}\text{~~as~~}k\to\infty.$$
Hence we see from $\sigma_0\delta_1\leq(s_1-s_0)\tau_0$ that
$$\bm{\phi}(z)\geq\bm{U}(z,z+(3s_1-s_0-1)\tau_0-\sigma_0\delta_1)\geq\bm{U}(z,z+(2s_1-1)\tau_0),\quad\forall z\in\Bbb{R},$$
that is, $\omega(\bm{\phi}_0)\geq\bm{V}^{(2s_1-1)\tau_0}(0,\cdot)$ in $\Bbb{R}$.
By Lemma \ref{convergence}, there exists $s_2\in[0,1]$ such that
$$\mathop{\lim}\limits_{n\to\infty}P^n(\bm{u}_0)=\omega(\bm{u}_0)=\bm{\zeta}(s_2)=\bm{V}^{(2s_2-1)\tau_0}(0,\cdot),$$
which leads to
$$\mathop{\lim}\limits_{t\to\infty}\lVert\bm{v}(t,\cdot;\bm{u}_0)-\bm{V}^{(2s_2-1)\tau_0}(0,\cdot)\rVert_{\mathcal C}=0.$$
Letting $\tilde{\tau}=\frac{(2s_2-1)\tau_0}{c}$, there is
$$\mathop{\lim}\limits_{t\to+\infty}\lVert\bm{v}(t,\cdot;\bm{u}_0)-\bm{V}^{c\tilde{\tau}}(t,\cdot)\rVert_{\mathcal C}=0,$$
that is,
$$\mathop{\lim}\limits_{t\to+\infty}\lVert\bm{u}(t,\cdot;\bm{u}_0)-\bm{U}(x,x+ct+c\tilde{\tau})\rVert_{\mathcal C}=0.$$
The proof is complete.
\end{proof}

\begin{corollary}
Let $\bm{U}(x,x+ct)$ and $\tilde{\bm{U}}(x,x+\tilde{c}t)$ be two pulsating fronts of system \eqref{u1u2-0-1}
connecting $\bm{0}$ to $\bm{1}$ with $c, \tilde{c}\neq0$.
Then there exists $\tilde{z}\in\Bbb{R}$ such that
$$\bm{U}(\cdot,\cdot+\tilde{z})=\tilde{\bm{U}}(\cdot,\cdot)\text{~~and~~}c=\tilde{c}.$$
\end{corollary}

\begin{proof}
Let $\bm{u}(t,x):=\bm{U}(x,x+ct)$ and $\tilde{\bm{u}}(t,x):=\tilde{\bm{U}}(x,x+\tilde{c}t)$.
In view of $\tilde{\bm{u}}(0,\cdot)\in\mathcal C_{\bm{1}}$,
$\mathop{\lim}\limits_{x\to-\infty}\tilde{\bm{u}}(0,x)=\bm{0}$ and
$\mathop{\lim}\limits_{x\to+\infty}\tilde{\bm{u}}(0,x)=\bm{1}$,
it follows that $\tilde{\bm{u}}(0,\cdot)$ satisfies \eqref{initial}.
The above Theorem \ref{GS} then yields that there exists some $\tilde{\tau}\in\Bbb{R}$ such that
$$\mathop{\lim}\limits_{t\to+\infty}\lVert\tilde{\bm{u}}(t,x)-\bm{U}(x,x+ct+c\tilde{\tau})\rVert_{\mathcal C}=0,$$
that is,
$$\mathop{\lim}\limits_{t\to+\infty}\lVert\tilde{\bm{U}}\left(z-ct,z+(\tilde{c}-c)t\right)-\bm{U}(z-ct,z+c\tilde{\tau})\rVert_{\mathcal C}=0.$$
Letting $t=nT$, then
\begin{equation}\label{U}
\mathop{\lim}\limits_{n\to+\infty}\tilde{\bm{U}}\left(z,z+(\tilde{c}-c)nT\right)=\bm{U}(z,z+c\tilde{\tau}),
\quad\text{uniformly in}\ z\in\Bbb{R}.
\end{equation}
If $\tilde{c}\neq c$, then for any fixed $z\in\Bbb{R}$,
$\mathop{\lim}\limits_{n\to+\infty}\tilde{\bm{U}}\left(z,z+(\tilde{c}-c)nT\right)=\bm{0}$ or $\bm{1}$,
which contradicts to \eqref{U}, and hence $\tilde{c}=c$. Then, we see from \eqref{U} that
$\tilde{\bm{U}}(z,z)=\bm{U}(z,z+c\tilde{\tau})$ for all $z\in\Bbb{R}$. That is,
$\bm{u}(\tilde{\tau},\cdot)=\tilde{\bm{u}}(0,\cdot)$ in $\Bbb{R}$, which implies that
$\bm{u}(t+\tilde{\tau},\cdot)=\tilde{\bm{u}}(t,\cdot)$ for any $t\in\Bbb{R}$,
namely, there exists $\tilde{z}:=c\tilde{\tau}$ such that
$\bm{U}(\cdot,\cdot+\tilde{z})=\tilde{\bm{U}}(\cdot,\cdot)$ in $\Bbb{R}^2$.
The proof is complete.
\end{proof}

\section*{Acknowledgments}
The first author was partially supported by FRFCU (lzujbky-2017-it59), the second author was partially supported by NSF of China (11671180, 11731005) and FRFCU (lzujbky-2017-ct01), and the third author partially supported by NSF of China (11301407) and  NSF of Shaanxi Province of China (2017JM1003).


\begin{thebibliography}{50}
\addtolength{\itemsep}{-1ex}

\bibitem{bao2013} X. Bao, Z.C. Wang,
Existence and stability of time periodic traveling waves for a periodic bistable Lotka-Volterra competition system,
\emph{J. Differential Equations } \textbf{255} (2013) 2402-2435.

\bibitem{bao2016} X. Bao, W.T. Li, W. Shen, Traveling wave solutions of Lotka-Volterra competition systems with nonlocal dispersal in periodic habitats,
\emph{J. Differential Equations} \textbf{260} (2016) 8590-8637.

\bibitem{berestycki2002} H. Berestycki, F. Hamel, Front propagation in periodic excitable media,
\emph{Comm. Pure Appl. Math. }\textbf{55} (2002) 949-1032.

\bibitem{berestycki20051} H. Berestycki, F. Hamel, L. Roques, Analysis of the periodically fragmented environment model: I-Species persistence,
\emph{J. Math. Biol.}  \textbf{51} (2005) 75-113.

\bibitem{Dockery1998} J. Dockery, V. Hutson, K. Mischaikow, M. Pernarowski,
The evolution of slow dispersal rates: a reaction diffusion model,
\emph{J. Math. Biol.} \textbf{37} (1998) 61-83.

\bibitem{ding2014} W. Ding, F. Hamel, X.Q Zhao, Bistable pulsating fronts for reaction-diffusion equations in a periodic habitat,
\emph{Indiana Univ. Math. J.} \textbf{66} (2017) 1189-1265.

\bibitem{du2016} L.J. Du, W.T. Li, J.B. Wang,
Invasion entire solutions in a time periodic Lotka-Volterra competition system with diffusion,
\emph{Math. Biosci. Eng. } \textbf{14} (2017) 1187-1213.

\bibitem{du20171} L.J. Du, W.T. Li, J.B. Wang, Asymptotic behavior of traveling fronts and entire solutions for a periodic bistable
competition-diffusion system, submitted.

\bibitem{du2017} L.J. Du, W.T. Li, S.L. Wu, Pulsating front-like entire solutions in a bistable Lotka-Volterra competition system with advection
in a periodic habitat, submitted.

\bibitem{fang2011} J. Fang, X.Q. Zhao, Bistable traveling waves for monotone semiflows with applications,
\emph{J. Eur. Math. Soc. } \textbf{17} (2011) 2243-2288.

\bibitem{fang2017} J. Fang, X. Yu, X.Q. Zhao, Traveling waves and spreading speeds for time-space periodic monotone systems,
\emph{J. Funct. Anal. }  \textbf{272} (2017) 4222-4262.

\bibitem{ferter1994} J. Furter, J. L$\acute{o}$pez-G$\acute{o}$mez, On the existence and uniqueness of coexistence
states for the Lotka-Volterra competition model with diffusion and spatially dependent coefficients,
\emph{Nonlinear Anal. } \textbf{25} (1995) 363-398.

\bibitem{Girardin2016} L. Girardin, Competition in periodic media: I -- Existence of pulsating fronts,
\emph{Discrete Contin. Dyn. Syst. Ser. B.} \textbf{22} (2016) 1341-1360.

\bibitem{guo2011} J.S. Guo, X. Liang, The minimal speed of traveling fronts for the Lotka-Volterra competition system,
\emph{J. Dynam. Differential Equations }\textbf{23} (2011) 353-363.

\bibitem{he2016} X. He, W.M. Ni, Global dynamics of the Lotka-Volterra competition-diffusion system: diffusion and spatial heterogeneity I,
\emph{Comm. Pure Appl. Math.} \textbf{69} (2016) 981-1014.

\bibitem{he20161} X. He, W.M. Ni, Global dynamics of the Lotka-Volterra competition-diffusion system with equal amount of total resources, II,
\emph{Calc. Var. Partial Differential Equations} \textbf{55} (2016) 1-20.

\bibitem{he2017}X. He, W.M. Ni, Global dynamics of the Lotka-Volterra competition-diffusion system with equal amount of total resources, III,
\emph{Calc. Var. Partial Differential Equations} \textbf{56} (2017) 132.

\bibitem{hess1991} P. Hess, Periodic-parabolic boundary value problems and positivity,
Longman Scientific and Technical, 1991.

\bibitem{hosono1989} Y. Hosono, The minimal speed of traveling fronts for a diffusive Lotka-Volterra competition model,
\emph{Bull. Math. Biol.}  \textbf{60} (1998) 435-448.

\bibitem{huang2011} W. Huang, M. Han, Nonlinear determinacy of minimum wave speed for a Lotka-Volterra competition model,
\emph{J. Differential Equations }\textbf{251} (2011) 1549-1561.

\bibitem{jin2008} Y. Jin, X.Q. Zhao, Bistable waves for a class of cooperative reaction-diffusion systems,
\emph{J. Biol. Dyn.} \textbf{2} (2008) 196-207.

\bibitem{kan-on1995} Y. Kan-on, Fisher wave fronts for the Lotka-Volterra competition model with diffusion,
\emph{Nonlinear Anal. TMA } \textbf{28} (1997) 145-164.

\bibitem{kong2015} L. Kong, N. Rawal, W. Shen,
Spreading speeds and linear determinacy for two species competition systems with nonlocal dispersal in periodic habitats,
\emph{Math. Model. Nat. Phenom.} \textbf{10} (2015) 113-141.

\bibitem{lam2012} K.Y. Lam, W. Ni, Uniqueness and complete dynamics of the Lotka-Volterra
competition diffusion system,
\emph{SIAM J. Appl. Math.} \textbf{72} (2012) 1695-1712.

\bibitem{Lewis2002} M.A. Lewis, B. Li, H.F. Weinberger, Spreading speeds and linear determinacy for two-species competition models,
\emph{J. Math. Biol.} \textbf{45} (2002) 219-233.

\bibitem{liang2006} X. Liang, Y. Yi, X.Q. Zhao, Spreading speeds and traveling waves for periodic evolution systems,
\emph{J. Differential Equations} \textbf{231} (2006) 57-77.

\bibitem{liang2007} X. Liang, X.Q. Zhao, Asymptotic speeds of spread and traveling waves for monotone semiflows with applications,
\emph{Comm. Pure Appl. Math.} \textbf{60} (2007) 1-40.

\bibitem{liang2010} X. Liang, X.Q. Zhao, Spreading speeds and traveling waves for abstract monostable evolution systems,
\emph{J. Funt. Anal.} \textbf{259} (2010) 857-903.

\bibitem{lou2006} Y. Lou, On the effects of migration and spatial heterogeneity on single and multiple species,
\emph{J. Differential Equations} \textbf{223} (2006) 400-426.

\bibitem{Lutscher2007} F. Lutscher, E. McCauley, M.A. Lewis, Spatial patterns and coexistence mechanisms in systems with unidirectional flow,
\emph{Theor. Popul. Biol.} \textbf{71} (2007) 267-277.

\bibitem{ma2016} M. Ma, X.Q. Zhao, Monostable waves and spreading speed for a reaction-diffusion model with seasonal succession,
\emph{Discrete Contin. Dyn. Syst. Ser. B.} \textbf{21} (2016) 591-606.

\bibitem{morita2009} Y. Morita, K. Tachibana, An entire solution to the Lotka-Volterra competition-diffusion equations,
\emph{SIAM J. Math. Anal. } \textbf{40} (2009) 2217-2240.

\bibitem{smith1995} H.L. Smith, Monotone dynamical systems: an introduction to the theory of competitive and cooperative Systems,
American Mathematical Society, Providence RI,1995.

\bibitem{lv2010} M. Wang, G. Lv,
Entire solutions of a diffusive and competitive Lotka-Volterra type system with nonlocal delays,
\emph{ Nonlinearity } \textbf{23} (2010) 1609-1630.

\bibitem{wang2012} Z.C. Wang, Traveling curved fronts in monotone bistable systems,
 \emph{Discrete Contin. Dyn. Syst.}  \textbf{32} (2012) 2339-2374.

\bibitem{weinberger2003} H.F. Weinberger, On spreading speeds and traveling waves for growth and migration models in a periodic habitat,
\emph{J. Math. Biol.}  \textbf{45} (2002) 511-548.

\bibitem{xu2004} D. Xu, X.Q. Zhao, Bistable waves in an epidemic model,
\emph{J. Dynam. Differential Equations }\textbf{16} (2004) 679-707.

\bibitem{Yu2017} X. Yu, X.Q. Zhao, Propagation phenomena for a reaction-advection-diffusion
competition model in a periodic habitat,
\emph{J. Dynam. Differential Equations }\textbf{29} (2017) 41-66.

\bibitem{zhang2012} Y. Zhang, X.Q. Zhao, Bistable traveling waves in competitive recursion systems,
\emph{J. Differential Equations}  \textbf{252} (2012) 2630-2647.

\bibitem{zhang2013} Y. Zhang, X.Q. Zhao, Bistable travelling waves for a reaction and diffusion model with seasonal succession,
\emph{Nonlinearity} \textbf{26} (2013) 691-709.

\bibitem{zhao2011} G. Zhao, S. Ruan, Existence, uniqueness and asymptotic stability of
time periodic traveling waves for a periodic Lotka-Volterra competition system with diffusion,
\emph{J. Math. Pures Appl. }\textbf{95} (2011) 627-671.

\bibitem{zhao2014} G. Zhao, S. Ruan,
Time periodic traveling wave solutions for periodic advection-reaction-diffusion systems,
\emph{J. Differential Equations }\textbf{257} (2014) 1078-1147.

\bibitem{zhao2003} X.Q. Zhao, Dynamical systems in population biology, Springer-Verlag, New York, 2003.


\end{thebibliography}
\end{document}